\documentclass[12pt,a4paper,leqno]{amsart}

\title[Propagation of exponential phase space singularities]{Propagation of exponential phase space singularities for Schr\"odinger equations with quadratic Hamiltonians}

\author[E. Carypis]{Evanthia Carypis}

\address{Department of Mathematics, University of Turin, Via Carlo Alberto 10, 10123 Torino (TO), Italy}

\email{evanthia.carypis@unito.it}

\author[P. Wahlberg]{Patrik Wahlberg}

\address{Department of Mathematics, Linn{\ae}us University, SE--35195 V{\"a}xj{\"o}, Sweden}

\email{patrik.wahlberg@lnu.se}

\usepackage{latexsym}
\usepackage{amsmath}
\usepackage{amssymb}
\usepackage{amsthm}
\usepackage{amsfonts}
\usepackage{mathrsfs}
\usepackage{calc}
\usepackage{cite}
\usepackage{color}

\numberwithin{equation}{section}          

\newtheorem{thm}{Theorem}
\numberwithin{thm}{section}

\newcommand{\rubrik}{}
\newtheorem{prop}[thm]{Proposition}
\newtheorem{cor}[thm]{Corollary}
\newtheorem{lem}[thm]{Lemma}

\theoremstyle{definition}

\newtheorem{defn}[thm]{Definition}
\newtheorem{example}[thm]{Example}

\theoremstyle{remark}

\newtheorem{rem}[thm]{Remark}              

\newcommand{\pd}[1] {\partial ^#1}

\newcommand{\ro}{\mathbf R}
\newcommand{\no}{\mathbf N}
\newcommand{\rr}[1]{\mathbf R^{#1}}
\newcommand{\nn}[1]{\mathbf N^{#1}}

\newcommand{\co}{\mathbf C}
\newcommand{\cc}[1]{\mathbf C^{#1}}

\newcommand{\nm}[2]{\Vert #1\Vert _{#2}}

\newcommand{\Ker}{\operatorname{Ker}}
\newcommand{\Ran}{\operatorname{Ran}}

\newcommand{\ep}{\varepsilon}
\newcommand{\fy}{\varphi}

\newcommand{\supp}{\operatorname{supp}}

\newcommand{\eabs}[1]{\langle #1\rangle}

\newcommand{\Sp}{\operatorname{Sp}}
\newcommand{\ssp}{\operatorname{sp}}
\newcommand{\Mp}{\operatorname{Mp}}
\newcommand{\GL}{\operatorname{GL}}
\newcommand{\charac}{\operatorname{char}}

\newcommand{\cS}{\mathscr{S}}

\newcommand{\cF}{\mathscr{F}}
\newcommand{\cK}{\mathscr{K}}

\newcommand{\wh}{\widehat}
\newcommand{\re}{{\rm Re} \, }
\newcommand{\im}{{\rm Im} \, }
\newcommand{\J}{\mathcal{J}}

\def\la{\langle}
\def\ra{\rangle}

\begin{document}

\begin{abstract}
We study propagation of phase space singularities for the initial value Cauchy problem for a class of Schr\"odinger equations. 
The Hamiltonian is the Weyl quantization of a quadratic form whose real part is non-negative. 
The equations are studied in the framework of projective Gelfand--Shilov spaces and their distribution duals. 
The corresponding notion of singularities is called the Gelfand--Shilov wave front set and means the lack of exponential decay in open cones in phase space. 
Our main result shows that the propagation is determined by the singular space of the quadratic form, just as in the framework of the Schwartz space, where the notion of singularity is the Gabor wave front set.  
\end{abstract}


\maketitle

\section{Introduction}\label{sec:intro}

The goal of this paper is to study propagation of singularities for 
the initial value Cauchy problem for a Schr\"odinger type equation
\begin{equation}\label{schrodingerequation}
\left\{
\begin{array}{rl}
\partial_t u(t,x) + q^w(x,D) u (t,x) & = 0,  \qquad t \geqslant 0, \quad x \in \rr d, \\
u(0,\cdot) & = u_0,  
\end{array}
\right.
\end{equation}
where $u_0$ is a Gelfand--Shilov distribution on $\rr d$, $q=q(x,\xi)$ is a quadratic form on the phase space $(x,\xi) \in T^* \rr d$ with $\re q \geqslant 0$, and $q^w(x,D)$ is a pseudodifferential operator in the Weyl quantization. 
This family of equations comprises the free Schr\"odinger equation where $q(x,\xi)= i |\xi|^2$, the harmonic oscillator where $q(x,\xi)= i (|x|^2 + |\xi|^2)$ and the heat equation where $q(x,\xi)= |\xi|^2$.

The problem has been studied in the space of tempered distributions \cite{Rodino2,Wahlberg1} where the natural notion of singularity is the Gabor wave front set. 
This concept of singularity is defined as the conical subset of the phase space $T^* \rr d \setminus \{0\}$ in which the short-time Fourier transform does not have rapid (superpolynomial) decay. The Gabor wave front set of a tempered distribution is empty exactly when the distribution is a Schwartz function so it measures deviation from regularity in the sense of both smoothness and decay at infinity comprehensively. 

In this work we study propagation of singularities for the equation \eqref{schrodingerequation} in the functional framework of the Gelfand--Shilov test function spaces and their dual distribution spaces. 
More precisely we use the projective limit (Beurling type) Gelfand--Shilov space $\Sigma_s(\rr d)$ for $s>1/2$ that consists of smooth functions satisfying
\begin{equation*}
\forall h>0 \ \ \exists C_h>0 : \quad |x^\alpha \pd \beta f(x)| \leqslant C_h h^{|\alpha+\beta|} (\alpha! \beta!)^s, \quad x \in \rr d, \quad \alpha,\beta \in \nn d. 
\end{equation*}
This means that $\Sigma_s(\rr d)$ is smaller than the Schwartz space, and hence its dual $\Sigma_s'(\rr d)$ is a space of distributions that contains the tempered distributions. 

The natural concept of phase space singularities in the realm of Gelfand--Shilov spaces is the $s$-Gelfand--Shilov wave front set.
The idea was introduced by H\"ormander \cite{Hormander1} under the name analytic wave front set for tempered distributions, and further developed by Cappiello and Schulz \cite{Cappiello1} and Cordero, Nicola and Rodino \cite{Cordero6} for Gelfand--Shilov distributions. These authors had an approach based on the inductive limit Gelfand--Shilov spaces as opposed to our concept that is based on the projective limit spaces. 

A concept similar to the $s$-Gelfand--Shilov wave front set has been studied by Mizuhara \cite{Mizuhara1}. 
This is the homogeneous wave front set of Gevrey order $s>1$. It is included in the $s$-Gelfand--Shilov wave front set. 
Propagation results for the homogeneous Gevrey wave front set are proved in \cite{Mizuhara1} for Schr\"odinger equations, albeit of a different type than ours. 

In this paper the $s$-Gelfand--Shilov wave front set $WF^s(u) \subseteq T^* \rr d \setminus \{0\}$ of $u \in \Sigma_s'(\rr d)$ for $s>1/2$
is defined as follows: 
$WF^s(u)$ is the complement in $T^* \rr d \setminus \{0\}$ of the set of $z_0 \in T^* \rr d \setminus \{0\}$ such that there exists an open conic set $\Gamma \subseteq T^* \rr d \setminus \{ 0 \}$ containing $z_0$, and
\begin{equation*}
\forall A>0: \ \sup_{z \in \Gamma} e^{A |z|^{1/s}} |V_\fy u(z) | < \infty.
\end{equation*}
Here $V_\fy u$ denotes the short-time Fourier transform defined by $\fy \in \Sigma_s (\rr d) \setminus \{0\}$. 
Thus the short-time Fourier transform decays like $e^{-A |z|^{1/s}}$ for any $A>0$ in an open cone around $z_0$. 
Note that this means that the decay can be close to but not quite like a Gaussian $e^{-A |z|^2}$, due to our assumption $s>1/2$. 

For a tempered distribution, the $s$-Gelfand--Shilov wave front set contains the Gabor wave front set, and thus gives an enlarged notion of singularity. 

Our main result on propagation of singularities for Schr\"odinger type equations goes as follows, where $e^{-t q^w(x,D)}$ denotes the solution operator (propagator) of the equation \eqref{schrodingerequation}.
Let $q$ be the quadratic form on $T^*\rr d$ defined by 
$q(x,\xi) = \la (x,\xi), Q(x,\xi) \ra$ 
and a symmetric matrix $Q \in \cc {2d \times 2d}$, $\re Q \geqslant 0$, $F=\J Q$
where 
\begin{equation}\label{Jdef}
\J =
\left(
\begin{array}{cc}
0 & I \\
-I & 0
\end{array}
\right) \in \rr {2d \times 2d}
\end{equation}
and $s > 1/2$. 
Then for $u_0 \in \Sigma_s'(\rr d)$
\begin{align*}
WF^s (e^{-t q^w(x,D)}u_0) 
& \subseteq  \left( e^{2 t \im F} \left( WF^s (u_0) \cap S \right) \right) \cap S, \quad t > 0,  
\end{align*}
where $S$ is the singular space
\begin{equation*}
S=\Big(\bigcap_{j=0}^{2d-1} \Ker\big[\re F(\im F)^j \big]\Big) \cap T^*\rr d \subseteq T^*\rr d 
\end{equation*}
of the quadratic form $q$. 
This result is verbatim the same as \cite[Theorem~5.2]{Rodino2} when $u_0$ is restricted to be a tempered distribution and when the $s$-Gelfand--Shilov wave front set is replaced by the Gabor wave front set (cf. \cite[Corollary~4.6]{Wahlberg1}). 
Thus it gives a new manifestation of the importance of the singular space for propagation of phase space singularities for the considered class of equations of Schr\"odinger type. 

The singular space has attracted much attention recently and occurs in several works on 
spectral and hypoelliptic properties of non-elliptic quadratic operators
\cite{Hitrik1,Hitrik2,Hitrik3,Pravda-Starov1,Pravda-Starov2,Viola1,Viola2}. 

The paper is organized as follows. 
Section \ref{sec:prelim} contains notations and background material on Gelfand--Shilov spaces, pseudodifferential operators, the short-time Fourier transform and the Gabor wave front set. 
Section \ref{sec:seminorm} gives a comprehensive discussion on three alternative families of seminorms for the projective Gelfand--Shilov spaces. 

In Section \ref{sec:defprop} we define the $s$-Gelfand--Shilov wave front set and deduce some properties: independence of the window function, symplectic invariance, behavior under tensor product and composition with surjective matrices, and microlocality with respect to pseudodifferential operators with certain symbols. In this process we show the continuity of metaplectic operators on $\Sigma_s(\rr d)$ and on $\Sigma_s'(\rr d)$.

Section \ref{sec:formulation} gives a brief discussion on the solution operator to the equation \eqref{schrodingerequation}, that is formulated for $u_0 \in L^2(\rr d)$ by means of semigroup theory. 
Exact propagation results are given for the case $\re Q = 0$. 
Section \ref{sec:proplinop} treats propagation of the $s$-Gelfand--Shilov wave front set for a class on linear operators, continuous on $\Sigma_s(\rr d)$ and uniquely extendible to continuous operators on $\Sigma_s'(\rr d)$.

In Section \ref{sec:oscint} we discuss shortly H\"ormander's oscillatory integrals with quadratic phase function, and 
we prove the inclusion of the $s$-Gelfand--Shilov wave front set of such an oscillatory integral in the intersection of its corresponding positive Lagrangian in $T^* \cc d$  with $T^* \rr d$. 
Section \ref{sec:kernelschrod} gives an account of H\"ormander's description of the Schwartz kernel of the propagator as an oscillatory integral and we show the continuity of the propagator on $\Sigma_s(\rr d)$ and on $\Sigma_s'(\rr d)$.

Finally in Section \ref{sec:propsing} we assemble the results of Sections \ref{sec:proplinop}, \ref{sec:oscint} and \ref{sec:kernelschrod} to prove our results on propagation of the $s$-Gelfand--Shilov wave front set for equations of the form \eqref{schrodingerequation}.

\section{Preliminaries}\label{sec:prelim}

\subsection{Notations and basic definitions}\label{subsec:notations}

The gradient of a function $f$ with respect to the variable $x \in \rr d$ is denoted by $f'_x$ and
the mixed Hessian matrix $(\partial_{x_i} \partial_{y_j} f)_{i,j}$ with respect to $x \in \rr d$ and $y \in \rr n$ is denoted $f_{x y}''$. 
The Fourier transform of $f \in \cS(\rr d)$ (the Schwartz space) is normalized as
\begin{equation*}
\mathscr{F} f(\xi) = \wh f(\xi) = \int_{\rr d} f(x) e^{- i \la x, \xi \ra} dx,
\end{equation*}
where $\la x, \xi \ra$ denotes the inner product on $\rr d$. The topological dual of $\mathscr S(\rr d)$ is the space of tempered distributions $\mathscr S'(\rr d)$.
As conventional $D_j = -i \partial_j$ for $1 \leqslant j \leqslant d$. 

We will make frequent use of the inequality
\begin{equation*}
|x+y|^{1/s} \leqslant C_s ( |x|^{1/s} + |y|^{1/s}), \quad x,y \in \rr d, 
\end{equation*}
where 
\begin{equation*}
C_s = 
\left\{
\begin{array}{ll}
1 & \mbox{if} \ s \geqslant 1 \\
2^{1/s-1} & \mbox{if} \ 0 < s < 1
\end{array}
\right. .
\end{equation*}
Since this inequality will be used only for $s>1/2$ we may use the cruder estimate $C_s=2$, which leads to the inequalities
\begin{align}
e^{A |x+y|^{1/s} } & \leqslant e^{2A |x|^{1/s}} e^{2A |y|^{1/s}}, \quad A >0, \quad x,y \in \rr d, \label{exppeetre1} \\
e^{- A |x+y|^{1/s} } & \leqslant e^{- \frac{A}{2} |x|^{1/s}} e^{A |y|^{1/s}}, \quad A >0, \quad x,y \in \rr d. \label{exppeetre2}
\end{align}

The Japanese bracket is $\eabs{x} = (1+|x|^2)^{1/2}$. 
For a positive measurable weight function $\omega$ defined on $\rr d$, 
the Banach space $L_{\omega}^1(\rr d)$ is endowed with the norm $\| f \|_{L_{\omega}^1}= \| f \omega \|_{L^1}$. 
The unit sphere in $\rr d$ is denoted $S_{d-1} = \{ x \in \rr d: \ |x|=1 \}$. 
For a matrix $A \in \rr {d \times d}$, $A \geqslant 0$  means that $A$ is positive semidefinite, and $A^t$ is the transpose. 
If $A$ is invertible then $A^{-t}$ denotes the inverse transpose. 
In estimates the notation $f (x) \lesssim g (x)$ understands that $f(x) \leqslant C g(x)$ holds for some constant $C>0$ that is uniform for all $x$ in the domain of $f$ and $g$. 
If $f(x) \lesssim g(x) \lesssim f(x)$ then we write $f(x) \asymp g(x)$. 

We denote the translation operator by $T_x f(y)=f(y-x)$, the modulation operator by $M_\xi f(y)=e^{i \la y, \xi \ra} f(y)$, $x,y,\xi \in \rr d$, and the
phase space translation operator by $\Pi(z) = M_\xi T_x$, $z=(x,\xi) \in \rr {2d}$.

\subsection{Gelfand--Shilov spaces}\label{subsec:GelfandShilov}

Let $h,s,t >0$. The space $\mathcal S_{t,h}^s(\rr d)$
is defined as all $f\in C^\infty (\rr d)$ such that
\begin{equation}\label{gfseminorm}
\nm f{\mathcal S_{t,h}^s}\equiv \sup \frac {|x^\alpha \partial ^\beta
f(x)|}{h^{|\alpha +\beta |}\alpha !^s\, \beta !^t}
\end{equation}
is finite. The supremum refers to all $\alpha ,\beta \in
\mathbf N^d$ and $x\in \rr d$. 
We set $\mathcal S_{s,h}=\mathcal S_{s,h}^s$.

The Banach space $\mathcal S_{t,h}^s$ increases with $h$, $s$ and $t$, 
and the embedding $\mathcal S_{t,h}^s\subseteq \cS$ holds for all $h,s,t >0$. 
If $s,t>1/2$, or $s=t =1/2$ and $h$ is sufficiently large, then $\mathcal
S_{t,h}^s$ contains all finite linear combinations of Hermite functions.

The \emph{Gelfand--Shilov spaces} $\mathcal S_{t}^s(\rr d)$ and
$\Sigma _{t}^s(\rr d)$ are the inductive and projective limits respectively
of $\mathcal S_{t,h}^s(\rr d)$ with respect to $h>0$. This means on the one hand
\begin{equation*}
\mathcal S_{t}^s(\rr d) = \bigcup _{h>0}\mathcal S_{t,h}^s(\rr d)
\quad \text{and}\quad \Sigma _{t}^s(\rr d) =\bigcap _{h>0}\mathcal S_{t,h}^s(\rr d).
\end{equation*}
On the other hand it means that the topology for $\mathcal S_{t}^s(\rr d)$ is the strongest topology such
that the inclusion $\mathcal S_{t,h}^s(\rr d) \subseteq \mathcal S_{t}^s(\rr d)$
is continuous for each $h>0$, 
and the topology for $\Sigma_{t}^s(\rr d)$ is the weakest topology such that 
the inclusion $\Sigma_{t}^s(\rr d) \subseteq \mathcal S_{t,h}^s(\rr d)$
is continuous for each $h>0$. 

The space $\Sigma _t^s(\rr d)$ is a Fr\'echet space with seminorms
$\| \cdot \|_{\mathcal S_{s,h}^t}$, $h>0$. 
The Gelfand--Shilov spaces are invariant under translation, modulation, dilation, linear coordinate transformations and tensor products. 
It holds $\Sigma _t^s(\rr d)\neq \{ 0\}$ if and only if $s,t>0$, $s+t \geqslant 1$ and $(s,t) \neq (1/2,1/2)$. 
We set $\mathcal S_{s}=\mathcal S_{s}^s$ and $\Sigma _{s}=\Sigma _{s}^s$.
Then $\mathcal S_s(\rr d)$ is zero when $s<1/2$, and $\Sigma _s(\rr d)$ is zero
when $s \le 1/2$. 
From now on we assume that $s>1/2$ when considering $\Sigma _s(\rr d)$.

The \emph{Gelfand--Shilov distribution spaces} $(\mathcal S_{t}^s)'(\rr d)$
and $(\Sigma _{t}^s)'(\rr d)$ are the projective and inductive limits
respectively of $(\mathcal S_{t,h}^s)'(\rr d)$.  This implies that
\begin{equation*}
(\mathcal S_{t}^s)'(\rr d) = \bigcap _{h>0}(\mathcal S_{t,h}^s)'(\rr d)\quad
\text{and}\quad (\Sigma _{t}^s)'(\rr d) =\bigcup _{h>0} (\mathcal S_{t,h}^s)'(\rr d).
\end{equation*}
The space $(\mathcal S_{t}^s)'(\rr d)$
is the topological dual of $\mathcal S_{t}^s(\rr d)$, and if $s>1/2$ then $(\Sigma _{t}^s)'(\rr d)$
is the topological dual of $\Sigma _{t}^s(\rr d)$ \cite{Gelfand1}.

In this paper we work with the spaces $\Sigma _s(\rr d)$ and $\Sigma _s'(\rr d) = (\Sigma _s^s)'(\rr d)$ for $s>1/2$. 
These spaces are embedded with respect to the Schwartz space and the tempered distributions as
\begin{equation*}
\Sigma _s(\rr d) \subseteq \cS(\rr d) \subseteq \cS'(\rr d) \subseteq \Sigma_s'(\rr d), \quad s > 1/2. 
\end{equation*}

For $s>1/2$ the (partial) Fourier transform extends 
uniquely to homeomorphisms on $\mathscr S'(\rr d)$, $\mathcal S_s'(\rr d)$
and $\Sigma _s'(\rr d)$, and restricts to 
homeomorphisms on $\mathscr S(\rr d)$, $\mathcal S_s(\rr d)$ and $\Sigma _s(\rr d)$.

\subsection{Pseudodifferential operators and the Gabor wave front set}\label{subset:pseudo}

Let $s>1/2$.
Given a window function $\fy \in \Sigma_s(\rr d) \setminus \{ 0 \}$, the short-time Fourier transform (STFT) \cite{Grochenig1} of 
$u \in \Sigma_s'(\rr d)$ is defined by
\begin{equation*}
V_\varphi u(x,\xi) = ( u,  M_\xi T_x \varphi ) = \cF(u \, T_x \overline{\fy}) (\xi), \quad x, \, \xi \in \rr d,
\end{equation*}
where $(\cdot,\cdot)$ denotes the conjugate linear action of $\Sigma_s'$ on $\Sigma_s$,
consistent with the inner product $(\cdot,\cdot)_{L^2}$ which is conjugate linear in the second argument.
The function $\rr {2d} \ni z \rightarrow V_\varphi u(z)$ is smooth.

Let $\fy,\psi \in \Sigma_s(\rr d) \setminus 0$.
By \cite[Theorem~2.5]{Toft1} we have 
\begin{equation}\label{STFTgrowth}
\forall u \in \Sigma_s'(\rr d) \quad \exists M \geqslant 0: \quad  |V_\varphi u (z)| \lesssim e^{M |z|^{1/s} }, \quad z \in \rr {2d}, 
\end{equation}
and by \cite[Lemma~11.3.3]{Grochenig1} we have 
\begin{equation}\label{STFTconvolution}
|V_\psi u(z)| \leqslant (2 \pi)^{-d} \| \fy \|_{L^2} |V_\fy u| * |V_\psi \fy| (z), \quad z \in \rr {2d},  
\end{equation}
where $V_\psi \fy \in \Sigma_s(\rr {2d})$. 

If $\fy \in \Sigma_s (\rr d)$ and $\| \fy \|_{L^2}=1$, 
the STFT inversion formula \cite[Corollary~11.2.7]{Grochenig1} reads
\begin{equation}\label{STFTrecon2}
(f,g ) = (2 \pi)^{-d} \int_{\rr {2d}} V_\varphi f(z) \, \overline{V_\varphi g(z)} \, dz, \quad f \in \Sigma_s'(\rr d),  \quad g \in \Sigma_s(\rr d).
\end{equation}

The Weyl quantization of pseudodifferential operators (cf. \cite{Folland1,Hormander0,Shubin1}) is the map from symbols $a \in \mathscr S(\rr {2d})$ to operators acting on $f \in \mathscr S(\rr d)$ defined by
\begin{equation}\nonumber
a^w(x,D) f(x) = (2 \pi)^{-d} \iint_{\rr {2d}} e^{i \la x-y,\xi \ra} a \left( \frac{x+y}{2},\xi \right)  \, f(y) \, dy \, d \xi.
\end{equation}
The conditions on $a$ and $f$ can be modified and relaxed in various ways.
The Weyl quantization can be formulated in the framework of Gelfand--Shilov spaces \cite{Cappiello2}. 
For certain symbols the operator $a^w(x,D)$ acts continuously on $\Sigma_s(\rr d)$ when $s>1/2$. 

If $a \in \Sigma_s'(\rr {2d})$ the Weyl quantization extends a continuous operator $\Sigma_s(\rr d) \rightarrow \Sigma_s'(\rr d)$
that satisfies
\begin{equation*}
(a^w(x,D) f, g) = (2 \pi)^{-d} (a, W(g,f) ), \quad f, g \in \Sigma_s(\rr d), 
\end{equation*}
where the cross-Wigner distribution is defined as 
\begin{equation*}
W(g,f) (x,\xi) = \int_{\rr d} g (x+y/2) \overline{f(x-y/2)} e^{- i \la y, \xi \ra} dy, \quad (x,\xi) \in \rr {2d}. 
\end{equation*}
We have $W(g,f) \in \Sigma_s(\rr {2d})$ when $f,g \in \Sigma_s(\rr d)$. 

We need the following symbol classes for pseudodifferential operators that act on $\cS(\rr d)$ in order to define the Gabor wave front set and explain its properties.

\begin{defn}\label{shubinclasses1}\cite{Shubin1}
For $m\in \ro$ the Shubin symbol class $G^m$ is the subspace of all
$a \in C^\infty(\rr {2d})$ such that for every
$\alpha,\beta \in \nn d$ 
\begin{equation*}
|\partial_x^\alpha \partial_\xi^\beta a(x,\xi)| 
\lesssim \langle (x,\xi) \rangle^{m-|\alpha|-|\beta|}, \quad (x,\xi)\in \rr {2d}. 
\end{equation*}
\end{defn}

\begin{defn}\label{hormanderclasses}\cite{Hormander0}
For $m\in \ro$, $0 \leqslant \rho \leqslant 1$, $0 \leqslant \delta < 1$, the H\"ormander symbol class $S_{\rho,\delta}^m$ is the subspace of all
$a \in C^\infty(\rr {2d})$ such that for every
$\alpha,\beta \in \nn d$ 
\begin{equation}\label{symbolestimate2}
|\partial_x^\alpha \partial_\xi^\beta a(x,\xi)| 
\lesssim \eabs{\xi}^{m - \rho|\beta| + \delta |\alpha|}, \quad (x,\xi)\in \rr {2d}.
\end{equation}
\end{defn}

Both $G^m$ and $S_{\rho,\delta}^m$ are Fr\'echet spaces with respect to their naturally defined seminorms. 

The following definition involves conic sets in the phase space $T^* \rr d \setminus 0 \simeq \rr {2d} \setminus 0$. 
A set is conic if it is invariant under multiplication with positive reals. 
Note the difference to the frequency-conic sets that are used in the definition of the (classical) $C^\infty$ wave front set \cite{Hormander0}. 

\begin{defn}\label{noncharacteristic2}
Given $a \in G^m$, a point $z_0 \in T^* \rr d \setminus 0$ is called non-characteristic for $a$ provided there exist $A,\ep>0$ and an open conic set $\Gamma \subseteq T^* \rr d \setminus 0$ such that $z_0 \in \Gamma$ and
\begin{equation*}
|a(z )| \geqslant \ep \eabs{z}^m, \quad z \in \Gamma, \quad |z| \geqslant A.
\end{equation*}
\end{defn}

The Gabor wave front set is defined as follows where $\charac (a)$ is the complement in $T^* \rr d \setminus 0$ of the set of non-characteristic points for $a$. 

\begin{defn}\label{wavefront1}
\cite{Hormander1}
If $u \in \mathscr S'(\rr d)$ then the Gabor wave front set $WF(u)$ is the set of all $z \in T^*\rr d \setminus 0$ such that $a \in G^m$ for some $m \in \ro$ and $a^w(x,D) u \in \mathscr S$ implies $z \in \charac(a)$.
\end{defn}

According to \cite[Proposition 6.8]{Hormander1} and \cite[Corollary 4.3]{Rodino1}, the Gabor wave front set can be characterized microlocally by means of the STFT as follows. 
If $u \in \mathscr S'(\rr d)$ and $\varphi \in \mathscr S(\rr d) \setminus 0$ then $z_0 \in T^*\rr d \setminus 0$ satisfies $z_0 \notin WF(u)$ if and only if there exists an open conic set $\Gamma_{z_0} \subseteq T^*\rr d \setminus 0$ containing $z_0$ such that
\begin{equation}\label{WFchar}
\sup_{z \in \Gamma_{z_0}} \eabs{z}^N |V_\varphi u(z)| < \infty \quad \forall N \geqslant 0.
\end{equation}

The most important properties of the Gabor wave front set include the following facts. 
Here the microsupport $\mu \supp (a)$ of $a \in G^m$ is defined as follows (cf. \cite{Schulz1}). 
For $z_0 \in T^* \rr d \setminus 0$ we have $z_0 \notin \mu \supp (a)$ if there exists an open cone $\Gamma \subseteq T^* \rr d \setminus 0$ containing $z_0$ such that 
\begin{equation*}
\sup_{z \in \Gamma} \eabs{z}^N |\pd \alpha a (z)| < \infty, \quad \alpha \in \nn {2d}, \quad N \geqslant 0. 
\end{equation*}
\begin{enumerate}

\item If $u \in \cS'(\rr d)$ then $WF(u) = \emptyset$ if and only if $u \in \cS (\rr d)$ \cite[Proposition 2.4]{Hormander1}.

\item If $u \in \mathscr S'(\rr d)$ and $a \in G^m$ then
\begin{align*}
WF( a^w(x,D) u) 
& \subseteq WF(u) \cap \mu \supp (a) \\ 
& \subseteq WF( a^w(x,D) u) \ \bigcup \ \charac (a). 
\end{align*}

\item If $a \in S_{0,0}^0$ and $u \in \cS'(\rr d)$ then by \cite[Theorem 5.1]{Rodino1}
\begin{equation}\label{microlocal2}
WF(a^w(x,D) u) \subseteq WF(u).
\end{equation}
In particular $WF(\Pi(z) u) = WF(u)$ for any $z \in \rr {2d}$.  

\end{enumerate}

As three basic examples of the Gabor wave front set we mention (cf. \cite[Examples~6.4--6.6]{Rodino1})
\begin{equation}\label{example1}
WF(\delta_x) = \{ 0 \} \times (\rr d \setminus 0 ), \quad x \in \rr d,  
\end{equation}
\begin{equation*}
WF(e^{i \la \cdot,\xi \ra}) = (\rr d \setminus 0) \times  \{ 0 \}, \quad \xi \in \rr d, 
\end{equation*}
and 
\begin{equation*}
WF(e^{i \la x, A x \ra/2 } ) = \{ (x, Ax): \, x \in \rr d \setminus 0 \}, \quad A \in \rr {d \times d} \quad \mbox{symmetric}. 
\end{equation*}

The canonical symplectic form on $T^* \rr d$ is
\begin{equation*}
\sigma((x,\xi), (x',\xi')) = \la x' , \xi \ra - \la x, \xi' \ra, \quad (x,\xi), (x',\xi') \in T^* \rr d.
\end{equation*}
With the matrix \eqref{Jdef} this can be expressed with the inner product on $\rr {2d}$ as
\begin{equation*}
\sigma((x,\xi), (x',\xi')) = \la \J (x,\xi), (x',\xi') \ra, \quad (x,\xi), (x',\xi') \in T^* \rr d. 
\end{equation*}

To each symplectic matrix $\chi \in \Sp(d,\ro)$ is associated an operator $\mu(\chi)$ that is unitary on $L^2(\rr d)$, and determined up to a complex factor of modulus one, such that
\begin{equation}\label{symplecticoperator}
\mu(\chi)^{-1} a^w(x,D) \, \mu(\chi) = (a \circ \chi)^w(x,D), \quad a \in \cS'(\rr {2d})
\end{equation}
(cf. \cite{Folland1,Hormander0}).
The operator $\mu(\chi)$ is a homeomorphism on $\mathscr S$ and on $\mathscr S'$.

The mapping $\Sp(d,\ro) \ni \chi \rightarrow \mu(\chi)$ is called the \emph{metaplectic representation} \cite{Folland1,Taylor1}.
It is in fact a representation of the so called $2$-fold covering group of $\Sp(d,\ro)$, which is called the metaplectic group 
and denoted $\Mp(d,\ro)$.
The metaplectic representation satisfies the homomorphism relation modulo a change of sign:
\begin{equation*}
\mu( \chi \chi') = \pm \mu(\chi ) \mu(\chi' ), \quad \chi, \chi' \in \Sp(d,\ro).
\end{equation*}

According to \cite[Proposition~2.2]{Hormander1} the Gabor wave front set is symplectically invariant as 
\begin{equation*}
WF( \mu(\chi) u) = \chi WF(u), \quad \chi \in \Sp(d, \ro), \quad u \in \cS'(\rr d).
\end{equation*}

The work \cite{Cordero5} contains a study of the propagation of the Gabor wave front set for linear Schr\"odinger equations, and \cite{Nicola2,Nicola1} 
contain studies of the same question for semilinear Schr\"odinger-type equations. 

\section{Seminorms on Gelfand--Shilov spaces}\label{sec:seminorm}

We need to work with several families of seminorms on $\Sigma_s(\rr d)$ for $s>1/2$ apart from the seminorms defined by \eqref{gfseminorm}.
The next result shows that there are three families of seminorms for $\Sigma _s(\rr d)$ that are each equivalent to the family of seminorms $\{\| f \|_{\mathcal S_{s,h}}, \ h > 0 \}$ defined by \eqref{gfseminorm}. 
The three families of seminorms are firstly $\{\| f \|_A', \ \| \wh f \|_B', \ A,B>0 \}$ where 
\begin{equation}\label{seminorm1}
\| f \|_A' = \sup_{x \in \rr d} e^{A |x|^{1/s}} |f(x)|, 
\end{equation}
secondly $\{| f |_A, \ A>0 \}$ where 
\begin{equation}\label{seminorm3}
| f |_A = \sup_{x \in \rr d, \ \beta \in \nn d} \frac{A^{|\beta|} e^{A |x|^{1/s}} |D^\beta f(x)|}{(\beta!)^s}, 
\end{equation}
and thirdly $\{\| f \|_A'', \ A>0 \}$ where 
\begin{equation}\label{seminorm2}
\| f \|_A'' = \sup_{z \in \rr {2d}} e^{A |z|^{1/s}} |V_\fy f(z)|, 
\end{equation}
for $\fy \in \Sigma_s(\rr d) \setminus 0$ fixed. 
It will turn out that the choice of $\fy \in \Sigma_s(\rr d) \setminus 0$  in the definition of the seminorms $\{\| f \|_A'', \ A>0 \}$ is arbitrary (see the proof of Proposition \ref{seminormequivalence}).

The essential arguments in the proof of the following proposition can be found in several places, e.g. \cite{Gelfand1,Grochenig3,Chung1,Nicola3,Toft2}. 
Nevertheless we prefer to give a detailed account since it is a cornerstone for our results, and in order to give a self-contained narrative. 

\begin{prop}\label{seminormequivalence}
Let $s>1/2$. Then 
\begin{equation}\label{seminorm2a}
\forall A,B>0 \quad \exists h>0: \ \| f \|_A' + \| \wh f \|_B' \lesssim \| f \|_{\mathcal S_{s,h}}, \quad f \in \Sigma_s(\rr d), 
\end{equation}
and
\begin{equation}\label{seminorm2b}
\forall h>0 \quad \exists A,B>0: \ \| f \|_{\mathcal S_{s,h}} \lesssim \| f \|_A' + \| \wh f \|_B' ,  \quad f \in \Sigma_s(\rr d). 
\end{equation}
Likewise 
\begin{equation}\label{seminorm4a}
\forall A>0 \quad \exists h>0: \ | f |_A \lesssim \| f \|_{\mathcal S_{s,h}},  \quad f \in \Sigma_s(\rr d), 
\end{equation}
and 
\begin{equation}\label{seminorm4b}
\forall h>0 \quad \exists A>0: \ \| f \|_{\mathcal S_{s,h}} \lesssim | f |_A ,  \quad f \in \Sigma_s(\rr d). 
\end{equation}
Finally
\begin{equation}\label{seminorm3a}
\forall A>0 \quad \exists h>0: \ \| f \|_A'' \lesssim \| f \|_{\mathcal S_{s,h}},  \quad f \in \Sigma_s(\rr d), 
\end{equation}
and 
\begin{equation}\label{seminorm3b}
\forall h>0 \quad \exists A>0: \ \| f \|_{\mathcal S_{s,h}} \lesssim \| f \|_A'',  \quad f \in \Sigma_s(\rr d). 
\end{equation}
\end{prop}

\begin{proof}
We start with \eqref{seminorm2a}. 
Let $f \in \Sigma_s(\rr d)$. 
From \eqref{gfseminorm} we have for any $h>0$
\begin{equation*}
|x^\alpha D^\beta f(x)| \leqslant \| f \|_{\mathcal S_{s,h}} (\alpha! \beta! )^s h^{|\alpha+\beta|}, \quad \alpha, \beta \in \nn d, \quad x \in \rr d.
\end{equation*}
This gives for any $n \in \no$ and any $\beta \in \nn d$
\begin{equation}\label{alphabetaestimate}
|x|^n |D^\beta f(x)| 
\leqslant d^{n/2} \max_{|\alpha|=n} |x^\alpha D^\beta f(x)|
\leqslant d^{n/2}
\| f \|_{\mathcal S_{s,h}} (n! \beta!)^s h^{n+|\beta|}, \quad x \in \rr d, 
\end{equation}
which in turn gives with $\beta=0$ and $A=2^{-1}s (d^{1/2} h)^{-1/s}$
\begin{align*}
\exp\left(\frac{A}{s}|x|^{1/s}\right) |f(x)|^{1/s}
& = \sum_{n=0}^\infty \frac{|x|^{n/s} |f(x)|^{1/s} (d^{1/2}  h)^{-n/s}}{n!}  \left(\frac{A (d^{1/2}  h)^{1/s}}{s}\right)^n \\
& \leqslant \| f \|_{\mathcal S_{s,h}}^{1/s} \sum_{n=0}^\infty 2^{-n}, \quad x \in \rr d. 
\end{align*}
For any $A>0$ we thus have 
\begin{equation}\label{seminorm2a1}
\| f \|_A' \lesssim \| f \|_{\mathcal S_{s,h_1}}, \quad f \in \Sigma_s(\rr d), 
\end{equation}
if $h_1=(s/(2A))^{s} d^{-1/2}$. 

Since the Fourier transform is continuous on $\Sigma_s(\rr d)$ we get from \eqref{seminorm2a1}
for any $B>0$
\begin{equation}\label{seminorm2a2}
\| \wh f \|_B' \lesssim \| \wh f \|_{\mathcal S_{s,h_0}} \lesssim \| f \|_{\mathcal S_{s,h_2}}, \quad f \in \Sigma_s(\rr d), 
\end{equation}
for some $h_0, h_2>0$.  Addition of \eqref{seminorm2a1} and \eqref{seminorm2a2} proves \eqref{seminorm2a} for $h=\min(h_1,h_2)$.  

The second and longer argument of this proof serves to prove \eqref{seminorm2b}. The argument follows closely that of the proof of \cite[Theorem~6.1.6]{Nicola3}. For completeness' sake we give the full details. 

First we deduce two estimates that are needed. From \eqref{seminorm2a} it follows that $\| f \|_A' < \infty$ and 
$\| \wh f \|_B' < \infty$ for any $A,B>0$ when $f \in \Sigma_s(\rr d)$. 
Thus for any $A>0$ we have 
\begin{align*}
\sum_{n=0}^\infty \frac{|x|^{n/s} |f(x)|^{1/s}}{n!} \left(\frac{A}{s} \right)^n
& = \exp\left( \frac{A}{s} |x|^{1/s} \right) |f(x)|^{1/s} \leqslant (\| f \|_A')^{1/s}, \quad x \in \rr d, 
\end{align*}
which gives the estimate
\begin{equation*}
|x|^{n} |f(x)| 
\leqslant \| f \|_A' (n!)^s \left(\frac{s}{A} \right)^{sn}, \quad x \in \rr d, \quad n \in \no. 
\end{equation*}
Using $|\alpha|! \leqslant d^{|\alpha|} \alpha!$ (cf. \cite[Eq.~(0.3.3)]{Nicola3}) this gives in turn
\begin{equation*}
|x^\alpha f(x)| 
\leqslant \| f \|_A' (\alpha!)^s \left( \frac{ds}{A}\right)^{s|\alpha|}, \quad x \in \rr d, \quad \alpha \in \nn d. 
\end{equation*}

Finally we take the $L^2$ norm and estimate for an integer $k>d/4$ with $\ep=4k-d>0$: 
\begin{equation}\label{L2est1}
\begin{aligned}
\| x^\alpha f \|_{L^2} 
& \lesssim \sup_{x \in \rr d} \eabs{x}^{(d+\ep)/2} |x^\alpha f(x)|
\lesssim \sup_{x \in \rr d, \ |\gamma| \leqslant 2 k} |x^{\alpha+\gamma} f(x)| \\
& \leqslant \| f \|_A' ((\alpha+\gamma)!)^s  \left( \frac{ds}{A}\right)^{s|\alpha+\gamma|} \\
& \lesssim \| f \|_A' (\alpha!)^s \left( \frac{2 d s}{A}\right)^{s|\alpha|}, \quad \alpha \in \nn d, 
\end{aligned}
\end{equation}
using $(\alpha+\gamma)! \leqslant 2^{|\alpha+\gamma|} \alpha! \gamma!$ (cf. \cite{Nicola3}) and considering $k$ a fixed parameter. 

From \eqref{L2est1}, $\| \wh f \|_B' < \infty$ for any $B>0$, and Parseval's theorem we obtain
\begin{equation}\label{L2est2}
\| D^\beta f \|_{L^2} 
= (2 \pi)^{-d/2} \| \xi^\beta \wh f \|_{L^2} 
\lesssim \| \wh f \|_B' (\beta!)^s \left( \frac{2 d s}{B}\right)^{s|\beta|}, \quad \beta \in \nn d. 
\end{equation}

Since $\| f \|_{A}' \leqslant \| f \|_{A+A_0}'$ when $A_0\geqslant 0$ and $A>0$ we may use $B=A$ when we now set out to prove \eqref{seminorm2b}.  
It suffices to assume $h \leqslant 1$. 
We have for $\alpha,\beta \in \nn d$ arbitrary and $f \in \Sigma_s(\rr d)$, 
using the Cauchy--Schwarz inequality, Parseval's theorem and the Leibniz rule
\begin{equation}\label{intermediateestimate1}
\begin{aligned}
|x^\alpha D^\beta f(x)| 
& = (2\pi)^{-d} \left| \int_{\rr d} \widehat{x^\alpha D^\beta f} (\xi) e^{i \la x, \xi \ra} d \xi \right|
\lesssim \| \eabs{\cdot}^{(d+\ep)/2} \widehat{x^\alpha D^\beta f} \|_{L^2} \\
& \lesssim \max_{|\gamma| \leqslant 2k} \| D^\gamma (x^\alpha D^\beta f) \|_{L^2} \\
& \lesssim \max_{|\gamma| \leqslant 2k} \sum_{\mu \leqslant \min(\alpha,\gamma) } \binom{\gamma}{\mu} \binom{\alpha}{\mu} \mu! \| x^{\alpha-\mu} D^{\beta + \gamma-\mu} f \|_{L^2}, \quad x \in \rr d. 
\end{aligned}
\end{equation}

In an intermediate step we rewrite the expression for the $L^2$ norm squared using integration by parts and estimate it as
\begin{align*}
& \| x^{\alpha-\mu} D^{\beta + \gamma-\mu} f \|_{L^2}^2 \\
& = |(D^{\beta + \gamma-\mu} f , x^{2\alpha-2\mu} D^{\beta + \gamma-\mu} f )| \\
& = |(f , D^{\beta + \gamma-\mu}  (x^{2\alpha-2\mu} D^{\beta + \gamma-\mu} f) )| \\
& \leqslant \sum_{\kappa \leqslant \min(\beta+\gamma-\mu,2\alpha-2\mu)} \binom{\beta+\gamma-\mu}{\kappa} \binom{2\alpha-2\mu}{\kappa} \kappa! |(x^{2\alpha-2\mu-\kappa} f, D^{2\beta + 2\gamma-2\mu-\kappa} f )| \\ 
& \leqslant \sum_{\kappa \leqslant \min(\beta+\gamma-\mu,2\alpha-2\mu)} \binom{\beta+\gamma-\mu}{\kappa} \binom{2\alpha-2\mu}{\kappa} \kappa! \| x^{2\alpha-2\mu-\kappa} f \|_{L^2}  \| D^{2\beta + 2\gamma-2\mu-\kappa} f \|_{L^2}. 
\end{align*}
Setting $h = 2^{2s+5/2} (2d s/A)^s$ and using \eqref{L2est1}, \eqref{L2est2} and $\kappa! = \kappa!^{2s -\delta}$ where $\delta=2s-1>0$, we get
\begin{align*}
& \| x^{\alpha-\mu} D^{\beta + \gamma-\mu} f \|_{L^2}^2 \\
& \lesssim 2^{2 |\alpha+\beta|} \|  f \|_A'  \| \wh f \|_A' \sum_{\kappa \leqslant \min(\beta+\gamma-\mu,2\alpha-2\mu)} \kappa! 
 ((2\alpha-2\mu-\kappa)! (2\beta + 2\gamma-2\mu-\kappa)!)^s \\
& \qquad \qquad \qquad \qquad \qquad \qquad \qquad \qquad \times \left(2^{-5/2-2s} h \right)^{|2\alpha-4\mu-2\kappa + 2\beta + 2\gamma|} \\
& \leqslant (2^{-3-4s}h^2)^{|\alpha+\beta|} \|  f \|_A'  \| \wh f \|_A' \sum_{\kappa \leqslant \min(\beta+\gamma-\mu,2\alpha-2\mu)}
 (\kappa!)^{-\delta} ((2\alpha-2\mu)! (2\beta + 2\gamma-2\mu)!)^s \\
& \qquad \qquad \qquad \qquad \qquad \qquad \qquad \qquad \times \left(2^{-5/2-2s} h \right)^{|-4\mu-2\kappa+2\gamma|} \\
& \lesssim (2^{-3-4s} h^2)^{|\alpha+\beta|} \|  f \|_A'  \| \wh f \|_A' \sum_{\kappa \leqslant \min(\beta+\gamma-\mu,2\alpha-2\mu)}
 ((2\alpha-2\mu)! (2\beta + 2\gamma-2\mu)!)^s \\
 & \lesssim (2^{-2-2s} h^2)^{|\alpha+\beta|} \|  f \|_A'  \| \wh f \|_A'  ((2\alpha-2\mu)! (2\beta-2\mu)!)^s
\end{align*}
since we have assumed $h \leqslant 1$, since $|\mu| \leqslant 2 k$ which is a fixed constant, 
and since 
\begin{equation*}
(\kappa!)^{-\delta} \left( 2^{-5/2-2s} h \right)^{- 2 |\kappa|} \leqslant \exp\left( \delta d \left( h^{-1} 2^{5/2+2s} \right)^{2/\delta} \right), \quad \kappa \in \nn d.
\end{equation*}

We insert this into \eqref{intermediateestimate1} which gives, using $\mu! \leqslant \mu!^{2s}$ and 
$$
(2(\alpha-\mu))! \leqslant 2^{2|\alpha|} ((\alpha-\mu)!)^2,
$$
\begin{align*}
& |x^\alpha D^\beta f(x)| \\
& \lesssim (2^{-1-s} h)^{|\alpha+\beta|} (\|  f \|_A'  \| \wh f \|_A')^{1/2} \max_{|\gamma| \leqslant 2k} \sum_{\mu \leqslant \min(\alpha,\gamma) } \binom{\gamma}{\mu} \binom{\alpha}{\mu} \mu!  ((2\alpha-2\mu)! (2\beta-2\mu)!))^{s/2} \\
& \lesssim ( 2^{-1} h)^{|\alpha+\beta|} (\|  f \|_A'  \| \wh f \|_A')^{1/2} \max_{|\gamma| \leqslant 2k} \sum_{\mu \leqslant \min(\alpha,\gamma) } \binom{\gamma}{\mu} \binom{\alpha}{\mu} (\alpha! \beta!)^{s} \\
& \lesssim h^{|\alpha+\beta|} (\alpha! \beta!)^{s} (\|  f \|_A' + \| \wh f \|_A'), \quad x \in \rr d, \quad \alpha,\beta \in \nn d.
\end{align*}
This finally proves \eqref{seminorm2b}, since for any $h>0$ we may take $A = 2^{3+5/2s} h^{-1/s} d s$. 

Next we show \eqref{seminorm4a} and \eqref{seminorm4b}. We start with \eqref{seminorm4a}.

From \eqref{alphabetaestimate} it follows that we have for $A>0$
\begin{align*}
\exp\left(\frac{A}{s}|x|^{1/s}\right) |D^\beta f(x)|^{1/s}
& = \sum_{n=0}^\infty \frac{|x|^{n/s} |D^\beta f(x)|^{1/s} (d^{1/2}  h)^{-n/s}}{n!}  \left(\frac{A (d^{1/2}  h)^{1/s}}{s}\right)^n \\
& \leqslant \| f \|_{\mathcal S_{s,h}}^{1/s} \beta! h^{|\beta|/s} \sum_{n=0}^\infty 2^{-n}, \quad x \in \rr d, \quad \beta \in \nn d, 
\end{align*}
provided $A \leqslant 2^{-1}s (d^{1/2} h)^{-1/s}$. 
Thus 
\begin{equation*}
e^{A|x|^{1/s}} |D^\beta f(x)| 
\lesssim \| f \|_{\mathcal S_{s,h}} (\beta!)^s h^{|\beta|}, \quad x \in \rr d, \quad \beta \in \nn d, 
\end{equation*}
which gives 
\begin{equation*}
|f|_A \lesssim \| f \|_{\mathcal S_{s,h}}, \quad f \in \Sigma_s(\rr d), 
\end{equation*}
provided $h \leqslant \min(A^{-1}, (s/(2A))^s d^{-1/2})$. 
Hence we have proved \eqref{seminorm4a}. 

We continue with the proof of \eqref{seminorm4b}. 
From \eqref{seminorm4a} we know that $|f|_A < \infty$ for any $A>0$ when $f \in \Sigma_s(\rr d)$. 
Hence for any $A>0$, $\beta \in \nn d$ and $x \in \rr d$
\begin{align*}
\sum_{n=0}^\infty \frac{|x|^{n/s} |D^\beta f(x)|^{1/s}}{n!} \left(\frac{A}{s} \right)^n
& = \exp\left( \frac{A}{s} |x|^{1/s} \right) |D^\beta f(x)|^{1/s} \leqslant | f |_A^{1/s} \beta! A^{-|\beta|/s}, 
\end{align*}
which gives 
\begin{align*}
|x|^{n} |D^\beta f(x)| 
& \leqslant | f |_A (n! \beta!)^s A^{-|\beta|} \left( \frac{s}{A} \right)^{sn}, \quad n \in \no, \quad \beta \in \nn d, \quad x \in \rr d, 
\end{align*}
and thus
\begin{align*}
|x^\alpha D^\beta f(x)| 
& \leqslant | f |_A (\alpha! \beta!)^s A^{-|\beta|} \left( \frac{ds}{A} \right)^{s|\alpha|}, \quad \alpha, \beta \in \nn d, \quad x \in \rr d. 
\end{align*}
From this it follows that 
\begin{equation*}
\| f \|_{\mathcal S_{s,h}} \lesssim |f|_A  , \quad f \in \Sigma_s(\rr d), 
\end{equation*}
for any $h>0$ provided $A \geqslant \max(h^{-1}, s d h^{-1/s})$. This proves \eqref{seminorm4b}. 

It remains to show \eqref{seminorm3a} and \eqref{seminorm3b}. We start with \eqref{seminorm3a}. 

Let $A>0$ and $\fy \in \Sigma_s(\rr d) \setminus 0$. 
We have for $f \in \Sigma_s(\rr d)$
\begin{align*}
|V_\fy f (x,\xi)| 
& = |\wh{f T_x \overline{\fy}} (\xi)|
\lesssim |\wh{f}| * |\wh{T_x \overline{\fy}}| (\xi)
= \int_{\rr d} |\wh{f}(\xi-\eta)| \, |\wh{\fy}(-\eta)| \, d \eta \\
& \lesssim \| \wh f \|_{8A}' \int_{\rr d} \exp(-8A|\xi-\eta|^{1/s}) \, |\wh{\fy}(-\eta)| \, d \eta \\
& \lesssim \| \wh f \|_{8A}'  \exp(-4A|\xi|^{1/s}) \int_{\rr d} \exp(8A|\eta|^{1/s}) \, |\wh{\fy}(-\eta)| \, d \eta \\
& \lesssim \| \wh f \|_{8A}'  \exp(-4A|\xi|^{1/s}), \quad x, \xi \in \rr d, 
\end{align*}
using \eqref{exppeetre2} and \eqref{seminorm2a}. 
From this estimate and $|V_\fy f (x,\xi)| = (2 \pi)^{-d} |V_{\wh \fy} \wh f (\xi,-x)|$ we also obtain 
\begin{equation*}
|V_\fy f (x,\xi)|  
\lesssim \| f \|_{8A}'  \exp(-4A|x|^{1/s}), \quad x, \xi \in \rr d. 
\end{equation*}
With the aid of \eqref{exppeetre1} we may conclude
\begin{align*}
e^{2 A |(x,\xi)|^{1/s}} |V_\fy f (x,\xi)|^2
& \leqslant e^{4 A |x|^{1/s}} |V_\fy f (x,\xi)| \ e^{4 A |\xi|^{1/s}} |V_\fy f (x,\xi)| \\
& \lesssim \| f \|_{8A}'  \ \| \wh f \|_{8A}'
\end{align*}
which gives 
\begin{equation*}
\| f \|_A'' \lesssim (\| f \|_{8A}' \ \| \wh f \|_{8A}')^{1/2} \lesssim \| f \|_{8A}' + \| \wh f \|_{8A}'.
\end{equation*}
Combining with \eqref{seminorm2a} we have proved \eqref{seminorm3a}. 

We now show \eqref{seminorm3b}.
For that purpose we use the strong version of the STFT inversion formula \eqref{STFTrecon2} and its Fourier transform, that is
\begin{align}
f(x) & = (2 \pi)^{-d} \int_{\rr {2d}} V_\fy f(y,\eta) M_\eta T_y \fy (x) \, dy \, d \eta, \label{STFTinv1} \\
\wh f(\xi) & = (2 \pi)^{-d} \int_{\rr {2d}} V_\fy f(y,\eta) T_\eta M_{-y} \wh \fy (\xi) \, dy \, d \eta, \label{STFTinv2}
\end{align}
where $f \in \Sigma_s(\rr d)$ and $\fy \in \Sigma_s(\rr d)$ satisfies $\| \fy \|_{L^2} = 1$. 
From \eqref{STFTinv1} we obtain for any $A>0$ 
\begin{align*}
e^{A |x|^{1/s}} |f(x)| 
& \lesssim \int_{\rr {2d}} |V_\fy f(y,\eta)| \, e^{A |x|^{1/s}} |\fy (x-y)| \, dy \, d \eta, \\
& \lesssim \| f \|_{3A}''  \int_{\rr {2d}} e^{-3A|(y,\eta)|^{1/s}} \, e^{A |x|^{1/s} -2 A |x-y|^{1/s}} dy \, d \eta, \\
& \lesssim \| f \|_{3A}''  \int_{\rr {2d}} e^{-3A|(y,\eta)|^{1/s}} \, e^{2 A |y|^{1/s}} dy \, d \eta, \\
& \lesssim \| f \|_{3A}'', \quad x \in \rr d, 
\end{align*}
which gives $\| f \|_A' \lesssim \| f \|_{3A}''$. 

From \eqref{STFTinv2} we obtain for any $A>0$ 
\begin{align*}
e^{A |\xi|^{1/s}} |\wh f(\xi)| 
& \lesssim \int_{\rr {2d}} |V_\fy f(y,\eta)| \, e^{A |\xi|^{1/s}} |\wh \fy (\xi-\eta)| \, dy \, d \eta, \\
& \lesssim \| f \|_{3A}''  \int_{\rr {2d}} e^{-3A|(y,\eta)|^{1/s}} \, e^{A |\xi|^{1/s} -2 A |\xi-\eta|^{1/s}} dy \, d \eta, \\
& \lesssim \| f \|_{3A}'', \quad \xi \in \rr d, 
\end{align*}
which gives $\| \wh f \|_A' \lesssim \| f \|_{3A}''$. Thus $\| f \|_A' + \| \wh f \|_A' \lesssim \| f \|_{3A}''$ so combining with \eqref{seminorm2b} we have proved \eqref{seminorm3b}. 

Finally we show that the seminorms $\{\| f \|_{A}'', \ A>0\}$ are equivalent to the same family of seminorms when the window function $\fy \in \Sigma_s(\rr d) \setminus 0$ is replaced by another function $\psi \in \Sigma_s(\rr d) \setminus 0$.
From \eqref{STFTconvolution} we obtain for $A>0$
\begin{align*}
e^{A |z|^{1/s}} |V_\psi f(z)| & \lesssim \int_{\rr {2d}} e^{A |z|^{1/s}} |V_\fy f(z-w)| \, |V_\psi \fy(w)| \, dw \\
& \lesssim \| f \|_{2A}'' \int_{\rr {2d}} e^{A |z|^{1/s} -2A|z-w|^{1/s}} \, |V_\psi \fy(w)| \, dw \\
& \lesssim \| f \|_{2A}'' \int_{\rr {2d}} e^{2A|w|^{1/s}} \, |V_\psi \fy(w)| \, dw \\
& \lesssim \| f \|_{2A}'', \quad z \in \rr {2d}, 
\end{align*}
using \eqref{seminorm3a} applied to $\fy \in \Sigma_s(\rr d)$, 
i.e. $\| \fy \|_{C}'' < \infty$ for all $C>0$. 
This proves the claim that the window function $\psi \in \Sigma_s(\rr d) \setminus 0$ gives seminorms equivalent to those of $\fy \in \Sigma_s(\rr d) \setminus 0$. 
\end{proof}

\section{Definition and properties of the $s$-Gelfand--Shilov wave front set}
\label{sec:defprop}

For $s>1/2$ and $u \in \Sigma_s' (\rr d)$ we define the $s$-Gelfand--Shilov wave front set $WF^s (u)$ 
as follows, modifying slightly Cappiello's and Schulz's \cite[Definition~2.1]{Cappiello1}. 
This concept is a coarsening of the Gabor wave front set $WF(u)$ in the sense that $WF(u) \subseteq WF^s(u)$ for all $s > 1/2$ and all $u \in \cS'(\rr d)$.

\begin{defn}\label{wavefronts}
Let $s > 1/2$, $\varphi \in \Sigma_s(\rr d) \setminus 0$ and $u \in \Sigma_s'(\rr d)$. 
Then $z_0 \in T^*\rr d \setminus 0$ satisfies $z_0 \notin WF^s (u)$ if there exists an open conic set $\Gamma_{z_0} \subseteq T^*\rr d \setminus 0$ containing $z_0$ such that for every $A>0$
\begin{equation*}
\sup_{z \in \Gamma_{z_0}} e^{A | z |^{1/s}} |V_\varphi u(z)| < \infty.
\end{equation*}
\end{defn}

It follows that $WF^s (u) = \emptyset$ if and only if $u \in \Sigma_s'(\rr d)$ satisfies
\begin{equation*}
|V_\varphi u(z)| \lesssim e^{-A | z |^{1/s}}, \quad z \in \rr {2d}, 
\end{equation*}
for any $A>0$. By Proposition \ref{seminormequivalence} (cf. \cite[Proposition~2.4]{Toft1}) this is equivalent to $u \in \Sigma_s(\rr d)$. 

The following lemma is needed in the proof of the independence of $WF^s (u)$ of the window function $\fy \in \Sigma_s(\rr d) \setminus 0$. 

\begin{lem}\label{convolutioninvariance}
Let $s>1/2$ and let $f$ be a measurable function on $\rr d$ that satisfies
\begin{equation}\label{polynomialbound1}
|f(x)| \lesssim e^{M | x |^{1/s}}, \quad x \in \rr d,
\end{equation}
for some $M \geqslant 0$. 
Suppose there exists a non-empty open conic set
$\Gamma \subseteq \rr d \setminus 0$ such that
\begin{equation}\label{conedecay1}
\sup_{x \in \Gamma} e^{A | x |^{1/s}} |f(x)| < \infty
\end{equation}
for all $A>0$. 
If
\begin{equation}\label{L1intersection}
g \in \bigcap_{A > 0} L_{\exp(A |\cdot |^{1/s})}^1(\rr d)
\end{equation}
then for any open conic set $\Gamma' \subseteq \rr d \setminus 0$ such that
$\overline{\Gamma' \cap S_{d-1}} \subseteq \Gamma$, we have
\begin{equation}\label{conedecay2}
\sup_{x \in \Gamma'} e^{A | x |^{1/s}} |f * g(x)| < \infty
\end{equation}
for all $A>0$. 
\end{lem}

\begin{proof}
By \eqref{exppeetre1} and the assumptions \eqref{polynomialbound1} and \eqref{L1intersection} 
we have
\begin{equation*}
|f * g (x)| \lesssim e^{2M | x |^{1/s}}, \quad x \in \rr d, 
\end{equation*}
so it suffices to assume $|x| \geqslant L$ for some large number $L>0$. 

Let $\ep>0$.
We estimate and split the convolution integral as
\begin{equation*}
|f * g(x)|
\leqslant  \underbrace{\int_{\eabs{y} \leqslant \ep \eabs{x}} |f(x-y)| \, | g (y)| \, d y}_{:= I_1}
+ \underbrace{\int_{\eabs{y} > \ep \eabs{x}} |f(x-y)| \, | g (y)| \, d y}_{:= I_2}.
\end{equation*}
Consider $I_1$.
Since $\eabs{y} \leqslant \ep \eabs{x}$ we have $x-y \in \Gamma$ if $x \in \Gamma'$, $|x| \geqslant 1$, and $\ep>0$ is chosen sufficiently small.
The assumptions \eqref{conedecay1}, \eqref{L1intersection}, and \eqref{exppeetre2} give
\begin{equation}\label{intuppsk1}
\begin{aligned}
I_1 & \lesssim \int_{\eabs{y} \leqslant \ep \eabs{x}} e^{- A |x+y|^{1/s} } |g (y)| \, d y
\lesssim e^{- \frac{A}{2} |x|^{1/s}} \int_{\rr d} e^{A |y|^{1/s}} |g(y)| \, d y \\
& \lesssim  e^{- \frac{A}{2} |x|^{1/s}}, \quad x \in \Gamma', \quad |x| \geqslant 1, 
\end{aligned}
\end{equation}
for any $A>0$. 
Next we estimate $I_2$ using \eqref{polynomialbound1} and $\eabs{y} > \ep \eabs{x}$. 
The latter inequality implies that $|y|^{1/s} \geqslant |x|^{1/s} \ep^{1/s}/2$ when $|x| \geqslant L$ if $L>0$ is sufficiently large. 
This gives for any $A>0$
\begin{equation}\label{intuppsk2}
\begin{aligned}
I_2 & \lesssim \int_{\eabs{y} > \ep \eabs{x}} e^{M |x-y|^{1/s} } \, |g(y) | \, dy \\
& \leqslant e^{2M |x|^{1/s} } \int_{\eabs{y} > \ep \eabs{x}}  e^{2M |y|^{1/s} } \, e^{-2 \ep^{-1/s}(2M+A) |y|^{1/s} } \, \, e^{2 \ep^{-1/s}(2M+A) |y|^{1/s} }\, |g(y) | \, dy \\
& \leqslant e^{(2M-2M-A) |x|^{1/s} } \int_{\rr d} e^{(2M+2 \ep^{-1/s} (2M +A))|y|^{1/s} }  \, |g(y) | \, dy \\
& \lesssim e^{-A |x|^{1/s} }, \quad x \in \rr d, \quad |x| \geqslant L, 
\end{aligned}
\end{equation}
again using \eqref{L1intersection}.
A combination of \eqref{intuppsk1} and \eqref{intuppsk2} proves \eqref{conedecay2} for $A>0$ arbitrary.
\end{proof}

Using Lemma \ref{convolutioninvariance} we show next that  Definition \ref{wavefronts} does not depend on the choice of the window function $\fy \in \Sigma_s(\rr d) \setminus 0$.

\begin{prop}\label{sGaborinvariance}
Suppose $s>1/2$ and $u \in \Sigma_s'(\rr d)$. The definition of the $s$-Gelfand--Shilov wave front set $WF^s(u)$ does not depend on the window function $\fy \in \Sigma_s(\rr d) \setminus 0$. 
\end{prop}

\begin{proof}
Let $\fy,\psi \in \Sigma_s(\rr d) \setminus 0$. 
By \eqref{STFTgrowth} we have for some $M \geqslant 0$
$$
|V_\varphi u (z)| \lesssim e^{M |z|^{1/s} }, \quad z \in \rr {2d}, 
$$
and by \eqref{STFTconvolution} we have 
\begin{equation*}
|V_\psi u(z)| \leqslant (2 \pi)^{-d} \| \fy \|_{L^2} |V_\fy u| * |V_\psi \fy| (z), \quad z \in \rr {2d}. 
\end{equation*}
By Proposition \ref{seminormequivalence} (cf. \cite[Theorem~2.4]{Toft1}) we have 
\begin{equation*}
|V_\psi \fy (z)| \lesssim e^{-A |z|^{1/s}}, \quad z \in \rr {2d}, 
\end{equation*} 
for any $A>0$, 
and hence 
\begin{equation*}
V_\psi \fy \in \bigcap_{A > 0} L_{\exp(A| \cdot |^{1/s})}^1(\rr {2d}). 
\end{equation*} 

From Lemma \ref{convolutioninvariance} we may now draw the following conclusion. 
If $|V_\fy u(z)|$ decays like $e^{-A |z|^{1/s} }$ for any $A>0$ in a conic set $\Gamma \subseteq T^* \rr d \setminus 0$ containing $z_0 \neq 0$ then 
we get decay like $e^{-A |z|^{1/s} }$ for any $A>0$ in a smaller cone containing $z_0$ for $|V_\psi u(z)|$. 
Hence, by symmetry, decay of order $e^{-A |z|^{1/s} }$ for any $A>0$ in an open cone around a point in $T^* \rr d \setminus 0$ happens simultaneously for $V_\fy u$ and $V_\psi u$. 
\end{proof}

The $s$-Gelfand--Shilov wave front set $WF^s (u)$ decreases when the index $s$ increases: 
\begin{equation}\label{sGaborinclusion}
t \geqslant s \quad \Longrightarrow \quad WF^t(u) \subseteq WF^s(u). 
\end{equation}

From $WF(u) \subseteq WF^s(u)$ for $u \in \cS'(\rr d)$ and \eqref{example1} we have for any $s > 1/2$
\begin{equation*}
WF^s(\delta_0) \supseteq \{ 0 \} \times (\rr d \setminus  0 ).  
\end{equation*}
To see the opposite inclusion we note that if $x_0 \in \rr d \setminus 0$ and $\xi_0 \in \rr d$ then $(x_0,\xi_0) \in \Gamma = \{(x,\xi) \in T^* \rr d \setminus 0: |\xi| < C |x| \}$ for some $C>0$, which is an open conic subset of $T^* \rr d \setminus 0$. 

Let $\fy \in \Sigma_s(\rr d) \setminus 0$ and let $\ep>0$. 
Since $|V_\fy \delta_0 (x, \xi)| = |\fy(-x)|$
it follows from Proposition \ref{seminormequivalence} that for any $A>0$
\begin{equation*}
\sup_{z \in \Gamma} e^{A | z |^{1/s}} |V_\varphi \delta_0(z)| 
\leqslant \sup_{z \in \Gamma} e^{2A (1+C^{1/s}) | x |^{1/s}} |\fy(-x)| < \infty.
\end{equation*} 
This shows that $(x_0,\xi_0) \notin WF^s(\delta_0)$, and proves 
\begin{equation*}
WF^s(\delta_0) \subseteq \{ 0 \} \times (\rr d \setminus 0 ).  
\end{equation*}
Hence
\begin{equation}\label{WFsdirac}
WF^s(\delta_0) = \{ 0 \} \times (\rr d \setminus 0 ), \quad s > 1/2.
\end{equation}

Next we show continuity of the metaplectic operators when they act on $\Sigma_{s}(\rr d)$ for $s>1/2$. 
We note that continuity of metaplectic operators acting on $\mathcal S_s (\rr d)$ for $s \geqslant 1$ is contained in
\cite[Proposition~3.5]{Cordero7}, and G.~Tranquilli \cite[Theorem~32]{Tranquilli1} has shown continuity on $\mathcal S_s (\rr d)$ for $s \geqslant 1/2$. Our proof is inspired by hers. 

\begin{prop}\label{metaplecticcont}
If $s>1/2$ and $\chi \in \Sp(d,\ro)$ then the metaplectic operator $\mu(\chi)$ acts continuously on $\Sigma_{s}(\rr d)$, and extends uniquely to a continuous operator on $\Sigma_s'(\rr d)$. 
\end{prop}

\begin{proof}
By \cite[Proposition~4.10]{Folland1} each matrix $\chi \in \Sp(d,\ro)$ is a finite product of matrices of the form
\begin{equation*}
\J, \quad 
\left(
  \begin{array}{cc}
  A^{-1} & 0 \\
  0 & A^{t}
  \end{array}
\right), 
\quad
\left(
  \begin{array}{cc}
  I & 0 \\
  B & I
  \end{array}
\right), 
\end{equation*}
for $A \in \GL(d,\ro)$ and $B \in \rr {d \times d}$ symmetric. 
To show that $\mu(\chi)$ is continuous on $\Sigma_{s}(\rr d)$ it thus suffices to show that 
$\mu(\chi)$ is continuous on $\Sigma_{s}(\rr d)$ when $\chi$ has each of these three forms. 

We have $\mu(\J) = (2 \pi)^{-d/2} \cF$, and $\mu(\chi) f(x) =  |A|^{1/2} f(Ax)$ when $A \in \GL(d,\ro)$ and
\begin{equation*}
\chi = 
\left(
  \begin{array}{cc}
  A^{-1} & 0 \\
  0 & A^{t}
  \end{array}
\right).  
\end{equation*}
The Fourier transform and linear coordinate transformations are continuous operators on $\Sigma_{s}(\rr d)$. 
Therefore it remains to prove that $\mu(\chi)$ is continuous on $\Sigma_{s}(\rr d)$ when $B \in \rr {d \times d}$ is symmetric and
\begin{equation}\label{chichirp}
\chi = 
\left(
  \begin{array}{cc}
  I & 0 \\
  B & I
  \end{array}
\right). 
\end{equation}
We have $\mu(\chi)f (x) = e^{i \la B x,x \ra/2} f(x)$ when \eqref{chichirp} holds (cf. \eqref{symplecticoperator} and \cite{Folland1}). 

Due to the continuity of coordinate transformations on $\Sigma_{s}(\rr d)$, it suffices 
to consider diagonal matrices $B$ with non-negative entries. 
By an induction argument applied to the seminorms \eqref{gfseminorm} it further suffices to
work in dimension $d=1$ and prove continuity on $\Sigma_{s}(\ro)$ of the multiplication operator $f \rightarrow g f$ for $g(x) = e^{i x^2/2}$. 

It may be confirmed by induction that 
for any $k \in \no$ we have $D^k g = p_k g$ where $p_k$ is the polynomial of order $k$
\begin{equation*}
p_k (x) = k! \sum_{m=0}^{\lfloor k/2  \rfloor} \frac{x^{k-2m} (-i)^m}{m!(k-2m)! 2^m}. 
\end{equation*}
Using $k! \leqslant 2^k (k-2m)! (2m)!$ we can estimate $|p_k(x)|$ as 
\begin{equation*}
|p_k (x)| 
\leqslant \sum_{m=0}^{\lfloor k/2  \rfloor} \frac{|x|^{k-2m} (2m)!}{m! 2^{m-k}} 
\leqslant \sum_{m=0}^{\lfloor k/2  \rfloor} |x|^{k-2m} m! 2^{m+k}.
\end{equation*}
By Leibniz' rule we have for any $A>0$ and any $B \geqslant A$
\begin{align*}
|D^n (g f)(x)| 
& \leqslant \sum_{k \leqslant n} \binom{n}{k} |p_k(x)| \, |D^{n-k} f(x)| \\
& \leqslant |f|_{2B} (n!)^s e^{-B|x|^{1/s}} A^{-n} \sum_{k \leqslant n} \binom{n}{k} |p_k(x)| e^{-B|x|^{1/s}} A^{k} \left( \frac{(n-k)!}{n!} \right)^s \\
& \leqslant |f|_{2B} (n!)^s e^{-A|x|^{1/s}} A^{-n} \sum_{k \leqslant n} \binom{n}{k} |p_k(x)| e^{-B|x|^{1/s}} A^{k} (k!)^{-s}, \\
& \qquad \qquad \qquad \qquad \qquad \qquad \qquad \qquad \qquad n \in \no, \quad x \in \ro.
\end{align*}

As an intermediate step we compute and estimate, using $m! = m!^{2s -\ep}$ where $\ep=2s-1>0$, 
\begin{align*}
|p_k(x)| A^{k} (k!)^{-s}
& \leqslant \sum_{m=0}^{\lfloor k/2  \rfloor} |x|^{k-2m} m! 2^{m+k} A^{k} (k!)^{-s} \\
& \leqslant \sum_{m=0}^{\lfloor k/2  \rfloor} 2^{m+k} A^{2m} m!  \left(\frac{(k-2m)!}{k!}\right)^s \left(\frac{(A|x|)^{(k-2m)/s}}{(k-2m)!}\right)^s \\
& \leqslant e^{s A^{1/s} |x|^{1/s}} \sum_{m=0}^{\lfloor k/2  \rfloor} 2^{m+k} \left(\frac{A^{2m/\ep}}{m!} \right)^\ep \left(\frac{m!^2(k-2m)!}{k!}\right)^s \\
& \leqslant e^{s A^{1/s} |x|^{1/s}} \sum_{m=0}^{\lfloor k/2  \rfloor} 2^{m+k} \left(\frac{A^{2m/\ep}}{m!} \right)^\ep \\
& \lesssim e^{s A^{1/s} |x|^{1/s}} e^{\ep A^{2/\ep}} 2^{2k}.
\end{align*}
If $B \geqslant sA^{1/s}$ we thus obtain
\begin{align*}
|D^n (g f)(x)| 
& \lesssim |f|_{2B} (n!)^s e^{-A|x|^{1/s}} (A/5)^{-n} \\
& \leqslant |f|_{2B} (n!)^s e^{-(A/5)|x|^{1/s}} (A/5)^{-n}, \quad n \in \no, \quad x \in \ro. 
\end{align*}
Hence we have shown that for any $A>0$ we have
\begin{equation*}
|g f|_A \lesssim |f|_B
\end{equation*}
for $B \geqslant \max(2 s(5A)^{1/s}, 10A)$. 
In view of Proposition \ref{seminormequivalence} this shows that the multiplication operator $f \rightarrow g f$ with $g(x) = e^{i x^2/2}$ is continuous on $\Sigma_{s}(\ro)$. 

Thus $\mu(\chi)$ is continuous on $\Sigma_s(\rr d)$ when $\chi$ has the form \eqref{chichirp} with $B$ symmetric. 
We may now conclude that $\mu(\chi)$ is continuous on $\Sigma_s(\rr d)$ for all $\chi \in \Sp(d, \ro)$. 

Finally the unique extension to a continuous operator on $\Sigma_s'(\rr d)$ follows from the facts that $\mu(\chi)$ is unitary and 
$\Sigma_s(\rr d)$ is dense in $\Sigma_s'(\rr d)$.
\end{proof} 

The combination of Propositions \ref{sGaborinvariance} and \ref{metaplecticcont} gives the symplectic invariance of the $s$-Gelfand--Shilov wave front set, as follows. 

\begin{cor}\label{symplecticWFs}
If $s>1/2$ then 
\begin{equation*}
WF^s(\mu(\chi) u) = \chi WF^s(u), \quad \chi \in \Sp(d,\ro), \quad u \in \Sigma_s' (\rr d). 
\end{equation*}
\end{cor}

\begin{proof}
By the proof of \cite[Lemma~3.7]{Wahlberg1} we have 
\begin{equation*}
\left| V_{\mu(\chi) \fy} \left( \mu(\chi) u \right)(\chi z) \right| = \left| V_\fy u (z) \right|
\end{equation*}
for $\fy \in \Sigma_s(\rr d)$, $u \in \Sigma_s'(\rr d)$, $\chi \in \Sp(d,\ro)$ and $z \in \rr {2d}$. 
By Proposition \ref{metaplecticcont} $\mu(\chi) \fy \in \Sigma_s(\rr d)$ so the 
result follows immediately from Proposition \ref{sGaborinvariance}. 
\end{proof}

\begin{example}
A combination of 
\eqref{WFsdirac} and $\mu(\J) = (2\pi)^{-d/2} \cF$ gives
\begin{equation}\label{WFsone}
WF^s( 1 ) = (\rr d \setminus 0) \times \{ 0 \}, \quad s> 1/2.
\end{equation}
If $B \in \rr {d \times d}$ is symmetric then $\chi$ defined by \eqref{chichirp} 
defines the metaplectic multiplication operator $\mu(\chi) = e^{i \la Bx,x \ra/2}$. 
Corollary \ref{symplecticWFs} combined with \eqref{WFsone} yields
\begin{equation}\label{WFschirp}
WF^s( e^{i \la Bx,x \ra/2} ) = \{ (x,Bx) : \ x \in \rr d \setminus 0 \}, \quad s > 1/2.
\end{equation}
\end{example}

Next we show a result on the $s$-Gelfand--Shilov wave front set of a tensor product. 
The corresponding result for the Gabor wave front set is \cite[Proposition~2.8]{Hormander1}.
With obvious modification of the proof given below you get an alternative proof of the latter result. 
Here we use the notation $x=(x',x'') \in \rr {m+n}$, $x' \in \rr m$, $x'' \in \rr n$. 

\begin{prop}\label{tensorWFs}
If $s>1/2$, $u \in \Sigma_s'(\rr m)$, and $v \in \Sigma_s'(\rr n)$ then 
\begin{align*}
& WF^s (u \otimes v) \subseteq \left( ( WF^s(u) \cup \{0\} ) \times ( WF^s(v) \cup \{0\} ) \right)\setminus 0 \\
& = \{ (x,\xi) \in T^* \rr {m+n} \setminus 0: \ (x',\xi') \in WF^s(u) \cup \{ 0 \}, \ (x'',\xi'') \in WF^s(v) \cup \{ 0 \} \} \setminus 0. 
\end{align*}
\end{prop}

\begin{proof}
Let $\fy \in \Sigma_s(\rr m) \setminus 0$ and $\psi \in \Sigma_s(\rr n) \setminus 0$. 
Suppose $(x_0,\xi_0) \in T^* \rr {m+n} \setminus 0$ does not belong to the set on the right hand side. 
Then either $(x_0',\xi_0') \notin WF^s(u) \cup \{ 0 \}$ or $(x_0'',\xi_0'') \notin WF^s(v) \cup \{ 0 \}$. 
For reasons of symmetry we may assume $(x_0',\xi_0') \notin WF^s(u) \cup \{ 0 \}$. 

Then $(x_0',\xi_0') \in \Gamma' \subseteq T^* \rr m \setminus 0$ where $\Gamma'$ is an open conic subset, and
\begin{equation*}
\sup_{(x',\xi') \in \Gamma'} e^{A | (x',\xi') |^{1/s}} |V_\varphi u(x',\xi'))| < \infty
\end{equation*}
for all $A>0$. 
Define for $C>0$ the open conic set
\begin{equation*}
\Gamma = \{ (x,\xi) \in T^* \rr {m+n}: \ (x',\xi') \in \Gamma', \ |(x'',\xi'')| < C |(x',\xi')| \} \subseteq T^* \rr {m+n} \setminus 0. 
\end{equation*}
Then $(x_0,\xi_0) \in \Gamma$ for $C>0$ sufficiently large since $(x_0',\xi_0') \neq 0$. 
By \eqref{STFTgrowth} we have for some $M \geqslant 0$
\begin{equation*}
|V_\psi v (z)| \lesssim e^{M |z|^{1/s} }, \quad z \in \rr {2d}.  
\end{equation*}
This gives for any $A>0$
\begin{align*}
\sup_{(x,\xi) \in \Gamma} e^{A | (x,\xi)|^{1/s}}& |V_{\fy \otimes \psi} u \otimes v (x,\xi)|
 = \sup_{(x,\xi) \in \Gamma} e^{A | (x,\xi)|^{1/s}} |V_\fy u (x',\xi')| \, |V_\psi v (x'',\xi'')| \\
& \lesssim \sup_{(x,\xi) \in \Gamma} e^{2A | (x',\xi')|^{1/s}+(2A+M)|(x'',\xi'')|^{1/s}} |V_\fy u (x',\xi')| \\
& \leqslant \sup_{(x',\xi') \in \Gamma'} e^{(2A+C^{1/s}(2A+M)) | (x',\xi')|^{1/s}} |V_\fy u (x',\xi')| \\
& < \infty. 
\end{align*}
It follows that $(x_0,\xi_0) \notin WF^s (u \otimes v)$. 
\end{proof}

We need in Section \ref{sec:oscint} the following result which is an adaptation of \cite[Proposition~2.3]{Hormander1} from the Gabor wave front set to the $s$-Gelfand--Shilov wave front set. Modified naturally the proof can be considered an alternative proof of the latter result. 

\begin{prop}\label{linsurj}
If $s>1/2$, $u \in \Sigma_s'(\rr d) \setminus 0$ and $A \in \rr {d \times n}$ is a surjective matrix, then
\begin{equation*}
WF^s(u \circ A) = \{ (x,A^t \xi) \in T^* \rr n \setminus 0: \ (Ax,\xi) \in WF^s(u) \} \cup \Ker A \setminus 0 \times \{ 0 \}.  
\end{equation*}
\end{prop}

\begin{proof}
Due to Corollary \ref{symplecticWFs}, and 
$\mu(\chi) f(x) =  |B|^{1/2} f(Bx)$ when $B \in \GL(d,\ro)$ and
\begin{equation*}
\chi = 
\left(
  \begin{array}{cc}
  B^{-1} & 0 \\
  0 & B^{t}
  \end{array}
\right),   
\end{equation*}
it suffices to assume $k := n - d > 0$ and $A = ( I_d \ \ 0)$ where $0 \in \rr {d \times k}$. 
We split variables as $x = (x',x'') \in \rr n$ with $x' \in \rr d$ and $x'' \in \rr k$. 
We need to prove 
\begin{equation}\label{WFequality1}
\begin{aligned}
& WF^s(u \otimes 1) \\
& = \{ (x; \xi',0) \in T^* \rr n \setminus 0: \ (x',\xi') \in WF^s(u) \} \cup \left( \{ 0_d \} \times \rr k \setminus 0 \times \{ 0_n \} \right)
\end{aligned}
\end{equation}
where we use the notation $0_n = 0\in \rr n$.

The inclusion
\begin{align*}
& WF^s(u \otimes 1) \\
& \subseteq \{ (x; \xi',0) \in T^* \rr n \setminus 0: \ (x',\xi') \in WF^s(u) \} \cup \left( \{ 0_d \} \times \rr k \setminus 0 \times \{ 0_n \} \right)
\end{align*}
is a particular case of Proposition \ref{tensorWFs}, combined with \eqref{WFsone}. 

To prove the opposite inclusion, we first show 
\begin{equation}\label{opposite1}
WF^s(u \otimes 1) \supseteq \{ 0_d \} \times \rr k \setminus 0 \times \{ 0_n \}. 
\end{equation}
Let $\fy \in \Sigma_s(\rr d) \setminus 0$ satisfy $(u,\fy) \neq 0$ and let $\psi \in \Sigma_s(\rr k) \setminus 0$ satisfy $\widehat \psi(0) \neq 0$. 
If $x'' \in \rr k \setminus 0$ then due to 
\begin{equation}\label{STFTtensor}
|V_{\fy \otimes \psi} u \otimes 1 (x,\xi)|
= |V_\fy u(x',\xi')| \ |\widehat \psi(-\xi'')|
\end{equation}
we have for any $t>0$ 
\begin{equation*}
|V_{\fy \otimes \psi} u \otimes 1 (t(0,x'';0)| = |V_\fy u(0,0)| \ |\widehat \psi(0)|
= |(u,\fy)| \ |\widehat \psi(0)| \neq 0.
\end{equation*}
Thus $V_{\fy \otimes \psi} u \otimes 1$ does not decay in any conical neighborhood of $(0,x'';0) \in T^* \rr n$,  
which proves \eqref{opposite1}. 

To prove \eqref{WFequality1} it thus suffices to show the inclusion
\begin{equation}\label{WFequality2}
WF^s(u \otimes 1) \supseteq \{ (x; \xi',0) \in T^* \rr n \setminus 0: \ (x',\xi') \in WF^s(u) \}. 
\end{equation}
Suppose $0 \neq (x_0; \xi_0',0) \notin WF^s(u \otimes 1)$. 
If $(x_0',\xi_0') = 0$ then $(x_0',\xi_0') \notin WF^s(u)$, so we may assume $(x_0',\xi_0') \neq 0$. 
We have $(x_0; \xi_0',0) \in \Gamma \subseteq T^* \rr n \setminus 0$ where $\Gamma$ is an open conic set such that
\begin{equation*}
\sup_{(x,\xi) \in \Gamma} e^{A | (x,\xi) |^{1/s}} |V_\fy u(x',\xi')| \ |\widehat \psi(-\xi'')| < \infty
\end{equation*}
for all $A>0$, cf. \eqref{STFTtensor}.

Define the open conic set
\begin{equation*}
\Gamma' = \{ (x',\xi') \in T^* \rr d \setminus 0: \ \exists x'' \in \rr k: \ (x',x'',\xi',0) \in \Gamma \} \subseteq T^* \rr d \setminus 0. 
\end{equation*}
Then $(x_0',\xi_0') \in \Gamma'$ since $(x_0',\xi_0') \neq 0$. 
Let $A>0$ be arbitrary.
Define the functions
\begin{align*}
f(x',\xi') & = e^{A | (x',\xi') |^{1/s}} |V_\fy u(x',\xi')| \ |\widehat \psi(0)|, \quad (x',\xi') \in T^* \rr d, \\
g(x,\xi) & = e^{A | (x',\xi') |^{1/s}} |V_\fy u(x',\xi')| \ |\widehat \psi(-\xi'')| \\
& \leqslant e^{A | (x,\xi) |^{1/s}} |V_\fy u(x',\xi')| \ |\widehat \psi(-\xi'')|, \quad (x,\xi) \in T^* \rr n. 
\end{align*}
For some sequence $(x_n',\xi'_n)_{n \in \no} \subseteq \Gamma'$, where for each $n \in \no$ there exists $x_n'' \in \rr k$ such that $(x_n',x_n'',\xi_n',0) \in \Gamma$, 
we have
\begin{align*}
\sup_{(x',\xi') \in \Gamma'} e^{A | (x',\xi') |^{1/s}} |V_\fy u(x',\xi')|  
& =  |\widehat \psi(0)|^{-1} \lim_{n \rightarrow \infty} f(x_n',\xi_n') \\
& =  |\widehat \psi(0)|^{-1} \lim_{n \rightarrow \infty} g(x_n',x_n'',\xi_n',0) \\
& \lesssim \sup_{(x,\xi) \in \Gamma} e^{A | (x,\xi) |^{1/s}} |V_\fy u(x',\xi')| \ |\widehat \psi(-\xi'')| \\
& < \infty. 
\end{align*}
This means that $(x_0',\xi_0') \notin WF^s(u)$ which proves \eqref{WFequality2}. 
\end{proof} 

Let $s>1/2$ and suppose $a \in C^\infty(\rr {2d})$ satisfies the estimates
\begin{equation}\label{symbolestimate0}
|\pd \alpha a(z)| \lesssim h^{|\alpha|} (\alpha!)^s, \quad \alpha \in \nn {2d}, \quad z \in \rr {2d}, 
\end{equation}
for all $h>0$. 
According to \cite[Theorem 3.4]{Cappiello2} $a^w(x,D)$ is then a continuous operator on $\Sigma_s (\rr d)$ that extends uniquely to a continuous operator on $\Sigma_s' (\rr d)$. In particular $WF^s (a^w(x,D) u)$ is well defined for $u \in \Sigma_s' (\rr d)$.
The next result shows that these pseudodifferential operators are microlocal with respect to the $s$-Gelfand--Shilov wave front set. First we need a lemma.

\begin{lem}\label{symbolSTFT}
If $\fy \in \Sigma_s(\rr {2d}) \setminus 0$ and $a \in C^\infty(\rr {2d})$ satisfies the estimates \eqref{symbolestimate0} 
for all $h>0$, 
then for any $A>0$
\begin{equation*}
\left| V_{\fy} a(x,\xi) \right| \lesssim e^{- A |\xi|^{1/s}}, \quad x \in \rr {2d}, \ \xi \in \rr {2d}. 
\end{equation*}
\end{lem}

\begin{proof}
We start by estimating a seminorm \eqref{gfseminorm} of $\overline{\fy} \, T_{-x} a$. 
From \eqref{symbolestimate0} we obtain for any $h>0$
\begin{align*}
|y^\alpha D_y^\beta (\overline{\fy(y)} a(y+x))|
& \leqslant \sum_{\gamma \leqslant \beta} \binom{\beta}{\gamma} |y^\alpha D^{\beta-\gamma} \fy(y)| \, |D^\gamma a(y+x)| \\
& \lesssim \| \fy \|_{\mathcal S_{s,h/2}} \sum_{\gamma \leqslant \beta} \binom{\beta}{\gamma} (h/2)^{|\beta-\gamma + \alpha+\gamma|} ((\beta-\gamma)! \gamma! \alpha!)^s \\
& \lesssim (h/2)^{|\alpha + \beta|} (\beta! \alpha!)^s \sum_{\gamma \leqslant \beta} \binom{\beta}{\gamma} \\
& \lesssim h^{|\alpha + \beta|} (\beta! \alpha!)^s, \quad x, y \in \rr {2d}, \quad \alpha,\beta \in \nn {2d}. 
\end{align*}
It follows that for any $h>0$ we have the estimate
\begin{equation*}
\| \overline{\fy} \, T_{-x} a \|_{\mathcal S_{s,h}} \leqslant C_h, \quad x \in \rr {2d},  
\end{equation*}
where $C_h>0$. Note that the estimate is uniform over $x \in \rr {2d}$.

By Proposition \ref{seminormequivalence}, or more precisely \eqref{seminorm2a}, we have for any $A>0$
\begin{equation*}
\| \wh{\overline{\fy} \, T_{-x} a} \|_A' \leqslant C_A, \quad x \in \rr {2d},  
\end{equation*}
for some $C_A>0$. 
This gives finally for any $A>0$
\begin{equation*}
\left| V_{\fy} a(x,\xi) \right| = |\wh{a T_x \overline{\fy}}(\xi)| = |\wh{\overline{\fy} \, T_{-x} a}(\xi)|
\lesssim e^{- A |\xi|^{1/s}}, \quad x \in \rr {2d}, \ \xi \in \rr {2d}. 
\end{equation*}
\end{proof}

\begin{prop}\label{microlocalWFs}
If $s>1/2$ and $a \in C^\infty(\rr {2d})$ satisfies the estimates \eqref{symbolestimate0}
for all $h>0$
then 
\begin{equation*}
WF^s (a^w(x,D) u) \subseteq WF^s (u), \quad u \in \Sigma_s'(\rr d). 
\end{equation*}
\end{prop}

\begin{proof}
Pick $\fy \in \Sigma_s(\rr d)$ such that $\| \fy \|_{L^2}=1$.
Denoting the formal adjoint of $a^w(x,D)$ by $a^w(x,D)^*$, \eqref{STFTrecon2} gives for $u \in \Sigma_s'(\rr d)$ and $z \in \rr {2d}$
\begin{align*}
V_\varphi (a^w(x,D) u) (z)
& = ( a^w(x,D) u, \Pi(z) \varphi ) \\
& = ( u, a^w(x,D)^* \Pi(z) \varphi ) \\
& = (2 \pi)^{-d} \int_{\rr {2d}} V_\varphi u(w) \, ( \Pi(w) \varphi,a^w(x,D)^* \Pi(z) \varphi ) \, dw \\
& = (2 \pi)^{-d} \int_{\rr {2d}} V_\varphi u(w) \, ( a^w(x,D) \, \Pi(w) \varphi,\Pi(z) \varphi ) \, dw \\
& = (2 \pi)^{-d} \int_{\rr {2d}} V_\varphi u(z-w) \, ( a^w(x,D) \, \Pi(z-w) \varphi,\Pi(z) \varphi ) \, dw.
\end{align*}
By e.g. \cite[Lemma 3.1]{Grochenig2} we have 
\begin{equation*}
|( a^w(x,D) \, \Pi(z-w) \varphi,\Pi(z) \varphi )|
= (2 \pi)^{-d} \left| V_\Phi a \left( z-\frac{w}{2}, \J w \right) \right|
\end{equation*}
where $\Phi$ is the Wigner distribution $\Phi = W(\fy,\fy) \in \Sigma_s(\rr {2d})$. 
Defining 
\begin{equation*}
g(w) = \sup_{z \in \rr {2d}} | ( a^w(x,D) \, \Pi(z-w) \varphi,\Pi(z) \varphi ) |, \quad w \in \rr {2d}, 
\end{equation*}
we thus obtain from Lemma \ref{symbolSTFT}
\begin{equation*}
g \in \bigcap_{A>0} L_{\exp(A | \cdot |^{1/s})}^1(\rr {2d})
\end{equation*}
and
\begin{align}\label{convolution1}
|V_\varphi (a^w(x,D) u) (z)|
& \lesssim |V_\varphi u| * g(z), \quad z \in \rr {2d}.
\end{align}
If $0 \neq z_0 \in T^*\rr d \setminus WF^s(u)$ then there exists an open conic set $\Gamma \subseteq T^* \rr d \setminus 0$ containing $z_0$ such that for all $A>0$
\begin{equation*}
\sup_{z \in \Gamma} e^{A |z|^{1/s}} |V_\varphi u(z)| < \infty.
\end{equation*}
By \eqref{STFTgrowth} we have for some $M > 0$
\begin{equation*}
|V_\varphi u (z)| \lesssim e^{M |z|^{1/s}}, \quad z \in \rr {2d}.
\end{equation*}
It now follows from \eqref{convolution1} and Lemma \ref{convolutioninvariance}
that for any open conic set $\Gamma'$ containing $z_0$ such that $\overline{\Gamma' \cap S_{2d-1}} \subseteq \Gamma$ we have for all $A > 0$
\begin{equation*}
\sup_{z \in \Gamma'} e^{A |z|^{1/s}} |V_\varphi (a^w(x,D) u) (z)| < \infty,
\end{equation*}
which proves that $z_0 \notin WF^s( a^w(x,D) u)$.
We have thus shown
\begin{equation*}
WF^s( a^w(x,D) u) \subseteq WF^s(u).
\end{equation*}
\end{proof} 

\begin{cor}
Let $s>1/2$ and $u \in \Sigma_s'(\rr d)$. 
For any $z \in \rr {2d}$ we have $WF^s( \Pi(z) u) = WF^s(u)$. 
\end{cor}

\begin{proof} 
Since $\Pi(-z) \Pi(z) = e^{i \la x, \xi \ra}$ for $z=(x,\xi) \in \rr {2d}$, 
it suffices to show $WF^s( \Pi(z) u) \subseteq WF^s(u)$. 
The latter inclusion follows from Proposition \ref{microlocalWFs} 
if we succeed in showing that the Weyl symbol for $\Pi(z)$ is smooth and satisfies \eqref{symbolestimate0} for any $h>0$. 

We have $\Pi(z) = a_z^w(x,D)$ where
\begin{equation*}
a_z(w) = e^{i\la x,\xi \ra/2 + i \la \J z, w \ra}, \quad z=(x,\xi), \quad w \in \rr {2d}  
\end{equation*}
(cf. the proof of \cite[Lemma~3.7]{Wahlberg1}). 
Thus 
\begin{align*}
|\pd \alpha a_z (w)| 
& \leqslant |z|^{|\alpha|} 
= h^{|\alpha|} (\alpha!)^s \left( \frac{(|z|/h)^{|\alpha|/s}}{\alpha!} \right)^s \\
& \leqslant h^{|\alpha|} (\alpha!)^s \left( \frac{(d(|z|/h)^{1/s})^{|\alpha|}}{|\alpha|!} \right)^s \\
& \leqslant h^{|\alpha|} (\alpha!)^s \exp(s d (|z|/h)^{1/s}), \quad \alpha \in \nn {2d}, \quad w \in \rr {2d}, 
\end{align*}
for any $h>0$. The estimates \eqref{symbolestimate0} are thus satisfied. 
\end{proof}

\section{Schr\"odinger equations and solution operators}
\label{sec:formulation}

As stated in Section \ref{sec:intro} the ultimate purpose of this paper is to prove results on propagation of the $s$-Gelfand--Shilov wave front set for the initial value Cauchy problem for a class of 
Schr\"odinger equations.
More precisely we study the equation
\begin{equation}\label{schrodeq}
\left\{
\begin{array}{rl}
\partial_t u(t,x) + q^w(x,D) u (t,x) & = 0, \\
u(0,\cdot) & = u_0, 
\end{array}
\right.
\end{equation}
where $s>1/2$, $u_0 \in \Sigma_s'(\rr d)$, $t \geqslant 0$ and $x \in \rr d$.
The Hamiltonian $q^w(x,D)$ has the quadratic form Weyl symbol
\begin{equation*}
q(x,\xi) = \la (x, \xi), Q (x, \xi) \ra, \quad x, \, \xi \in \rr d, 
\end{equation*}
where $Q \in \cc {2d \times 2d}$ is a symmetric matrix with $\re Q \geqslant 0$. 
The special case $\re Q = 0$ will admit to study the equation for $t \in \ro$ instead of $t \geqslant 0$. 

According to \cite[Theorem 3.4]{Cappiello2} $q^w(x,D)$ extends to a continuous operator on $\Sigma_s'(\rr d)$, and we will 
later prove that also the solution operator is continuous on $\Sigma_s'(\rr d)$ for each $t \geqslant 0$ (see Corollary \ref{propagatorcontGF}). 

The \emph{Hamilton map} $F$ corresponding to $q$ is defined by 
\begin{equation*}
\sigma(Y, F X) = q(Y,X), \quad X,Y \in \rr {2d}, 
\end{equation*}
where $q(Y,X)$ is the bilinear polarized version of the form $q$, i.e. $q(X,Y)=q(Y,X)$ and $q(X,X)=q(X)$. 
The Hamilton map $F$ is the matrix 
\begin{equation*}
F = \J Q \in \cc {2d \times 2d} 
\end{equation*}
where $\J$ is the matrix \eqref{Jdef}. 

For $u_0 \in L^2(\rr d)$ the equation \eqref{schrodeq} is solved for $t \geqslant 0$ by 
\begin{equation*}
u(t,x) = e^{-t q^w(x,D)} u_0(x)
\end{equation*}
where the solution operator (propagator) $e^{-t q^w(x,D)}$ is the contraction semigroup that is generated by the operator $-q^w(x,D)$. 
Contraction semigroup means a strongly continuous semigroup with $L^2$ operator norm  $\leqslant 1$ for all $t \geqslant 0$ (cf. \cite{Yosida1}). 
The reason why $- q^w(x,D)$, or more precisely the closure $M_{-q}$ as an unbounded linear operator in $L^2(\rr d)$ of the operator $- q^w(x,D)$ defined on $\cS(\rr d)$, generates such a semigroup is explained in \cite[pp. 425--26]{Hormander2}. 
The contraction semigroup property is a consequence of $M_{- q}$ and its adjoint $M_{-\overline q}$ being \emph{dissipative} operators \cite{Yosida1}. 
For $M_{-q}$ this means
\begin{equation*}
\re (M_{-q} u,u) = (M_{-\re q} u,u) \leqslant 0, \quad u \in D(M_{- q}), 
\end{equation*}
$D(M_{-q}) \subseteq L^2(\rr d)$ denoting the domain of $M_{-q}$, which follows from the assumption $\re Q \geqslant 0$. 
Note the feature $M_{-\overline q} = M_{-q}^*$ that holds for the Weyl quantization. 

Our objective is the propagation of the $s$-Gelfand--Shilov wave front set with $s>1/2$ for the Schr\"odinger propagator $e^{-t q^w(x,D)}$. 
This means that we seek inclusions for 
\begin{equation*}
WF^s(e^{-t q^w(x,D)} u_0)
\end{equation*}
in terms of $WF^s(u_0)$, $F$ and $t \geqslant 0$ for $u_0 \in \Sigma_s'(\rr d)$. 

If $\re Q=0$ then the propagator is given by means of the metaplectic representation. 
To wit, if 
$\re Q=0$ then $e^{-t q^w(x,D)}$ is a group of unitary operators, and we have by \cite[Theorem 4.45]{Folland1}
\begin{equation*}
e^{-t q^w(x,D)} = \mu(e^{-2 i t F}), \quad t \in \ro. 
\end{equation*}
In this case $F$ is purely imaginary and $i F \in \ssp(d,\ro)$, the symplectic Lie algebra, which implies that $e^{-2 i t F} \in \Sp(d,\ro)$ for any $t \in \ro$ \cite{Folland1}.
According to Corollary \ref{symplecticWFs} we thus have if $s > 1/2$
\begin{equation*}
WF^s(e^{-t q^w(x,D)} u_0) = e^{-2 i t F} WF^s(u_0), \quad t \in \ro, \quad u_0 \in \Sigma_s'(\rr d). 
\end{equation*}

The propagation of the $s$-Gelfand--Shilov wave front set is thus exact when $\re Q=0$. 
In the rest of the paper we study the more general assumption $\re Q \geqslant 0$.
Under this assumption we will show in Section \ref{sec:kernelschrod} that the propagator $e^{-t q^w(x,D)}$ is a continuous operator on $\Sigma_s(\rr d)$ and extends uniquely to a continuous operator on $\Sigma_s'(\rr d)$ when $s>1/2$.

\section{Propagation of the $s$-Gelfand--Shilov wave front set for certain linear operators}\label{sec:proplinop}

In this section we prepare for the results on propagation of the $s$-Gelfand--Shilov wave front set for $e^{-t q^w(x,D)}$ in Section \ref{sec:propsing}.
We show propagation of singularities for linear operators in terms of their Schwartz kernels. 

For $s>1/2$ a kernel $K \in \Sigma_s'(\rr {2d})$ defines a continuous linear map $\cK: \Sigma_s (\rr d) \rightarrow \Sigma_s'(\rr d)$ by
\begin{equation}\label{kernelop}
(\cK f, g) = (K, g \otimes \overline f), \quad f,g \in \Sigma_s(\rr d). 
\end{equation}
Let $\fy \in \Sigma_s(\rr d)$ satisfy $\| \fy \|_{L^2} = 1$ and set $\Phi = \fy \otimes \fy \in \Sigma_s(\rr {2d})$. 
By \cite[Lemma~4.1]{Wahlberg1} we have for $u,\psi \in \Sigma_s(\rr d)$
\begin{equation}\label{TSTFT}
\begin{aligned}
(\cK u, \psi) 
= (2 \pi)^{-2d} \int_{\rr {4d}} V_\Phi K(x,y,\xi,-\eta) \, \overline{V_\fy \psi (x,\xi)} \, V_{\overline \fy} u(y,\eta) \, dx \, dy \, d \xi \, d \eta. 
\end{aligned}
\end{equation}

In the following results we need a definition from \cite{Hormander1}, adapted from the Gabor to the $s$-Gelfand--Shilov wave front set. 
For $K \in \Sigma_s'(\rr {2d})$ we define
\begin{equation}\label{WFproj}
\begin{aligned}
WF_1^s(K) & = \{ (x,\xi) \in T^* \rr d: \ (x, 0, \xi, 0) \in WF^s(K) \} & \subseteq T^* \rr d \setminus 0, \\
WF_2^s(K) & = \{ (y,\eta) \in T^* \rr d: \ (0, y, 0, -\eta) \in WF^s(K) \} & \subseteq T^* \rr d \setminus 0. 
\end{aligned}
\end{equation}

\begin{lem}\label{WFkernelaxes}
If $s>1/2$, $K \in \Sigma_s'(\rr {2d})$ and $WF_1^s(K) =  WF_2^s(K) = \emptyset$, then there exists $C > 1$ such that 
\begin{equation*}
WF^s (K) \subseteq \{ (x,y,\xi,\eta) \in T^* \rr {2d}: \ C^{-1} |(x,\xi)| <  |(y,\eta)| < C |(x,\xi)| \}. 
\end{equation*}
\end{lem}

\begin{proof}
Suppose 
\begin{equation*}
WK^s (K) \subseteq \{ (x,y,\xi,\eta) \in T^* \rr {2d}: \ |(y,\eta)| < C |(x,\xi)| \}
\end{equation*}
does not hold for any $C>0$. Then for each $n \in \no$ there exists $(x_n,y_n,\xi_n,\eta_n) \in WF^s(K)$ such that $|(y_n,\eta_n)| \geqslant n |(x_n,\xi_n)|$. We may assume that $|(x_n,y_n,\xi_n,\eta_n)| = 1$ for each $n \in \no$ since $WF^s(K)$ is conic. 
Thus $(x_n,\xi_n) \rightarrow 0$ as $n \rightarrow \infty$. 
Passing to a subsequence (without change of notation) and using the closedness of $WF^s(K)$ gives
\begin{equation*}
(x_n,y_n,\xi_n,\eta_n) \rightarrow (0,y,0,\eta) \in WF^s(K), \quad n \rightarrow \infty, 
\end{equation*}
for some $(y,\eta) \in S_{2d-1}$. This implies $(y,-\eta) \in WF_2^s(K)$ which is a contradiction. 

Similarly one shows 
\begin{equation*}
WK^s (K) \subseteq \{ (x,y,\xi,\eta) \in T^* \rr {2d}: \ |(x,\xi)| < C |(y,\eta)| \}
\end{equation*}
for some $C>0$ using $WF_1^s(K) = \emptyset$. 
\end{proof}

In the next result we use the conventional notation (cf. \cite{Hormander1,Hormander2}) for the reflection operator in the fourth $\rr d$ coordinate on $\rr {4d}$
\begin{equation*}
(x,y,\xi,\eta)' = (x,y,\xi,-\eta), \quad x,y,\xi,\eta \in \rr d. 
\end{equation*}

\begin{lem}\label{STFTkernelformula}
Suppose $s>1/2$, $K \in \Sigma_s'(\rr {2d})$ and $WF_1^s(K) =  WF_2^s(K) = \emptyset$. 
Suppose that the linear map $\cK: \Sigma_s (\rr d) \rightarrow \Sigma_s'(\rr d)$ defined by \eqref{kernelop}
is continuous $\cK: \Sigma_s (\rr d) \rightarrow \Sigma_s(\rr d)$ and extends uniquely to a continuous linear operator 
$\cK: \Sigma_s' (\rr d) \rightarrow \Sigma_s'(\rr d)$. 
If $\fy \in \Sigma_s(\rr d)$ satisfies $\| \fy \|_{L^2} = 1$ and $\Phi = \fy \otimes \fy$
then \eqref{TSTFT} extends to $u \in \Sigma_s'(\rr d)$ and $\psi \in \Sigma_s(\rr d)$. 
\end{lem}

\begin{proof}
Let $\fy \in \Sigma_s(\rr d)$ satisfy $\| \fy \|_{L^2} = 1$ and let $u \in \Sigma_s'(\rr d)$. 
By \eqref{STFTgrowth} we have for some $M \geqslant 0$
\begin{equation}\label{STFTupperbound1}
|V_\varphi u (z)| \lesssim e^{M |z|^{1/s}}, \quad z \in \rr {2d}.
\end{equation}
Define for $n \in \no$
\begin{equation*}
u_n = (2 \pi)^{-d} \int_{|z| \leqslant n} V_\fy u(z) \Pi(z) \fy \, dz. 
\end{equation*}

In order to verify $u_n \in \Sigma_s(\rr d)$ we use the seminorms \eqref{seminorm2} for $A>0$. 
We have for any $A>0$
\begin{align*}
e^{A|w|^{1/s}} |V_\fy u_n (w)| 
& \lesssim \int_{|z| \leqslant n} |V_\fy u(z)| \, e^{A|w|^{1/s}} |V_\fy \fy(w-z)| \, dz \\
& \lesssim \int_{|z| \leqslant n} e^{M |z|^{1/s} + A|w|^{1/s} - 2 A |w-z|^{1/s}} \, dz \\
& \lesssim \int_{|z| \leqslant n} e^{(M+2A) |z|^{1/s}} \, dz \\
& \lesssim 1, \quad w \in \rr {2d}. 
\end{align*}
It follows that $u_n \in \Sigma_s(\rr d)$ for $n \in \no$. 

To prove that $u_n \rightarrow u$ in $\Sigma_s'(\rr d)$ as $n \rightarrow \infty$ we pick $\psi \in \Sigma_s(\rr d)$. 
From \eqref{STFTupperbound1}, the estimate 
(cf. Proposition \ref{seminormequivalence} and \cite[Theorem~2.4]{Toft1})
\begin{equation}\label{STFTupperbound2}
|V_\varphi \psi (z)| \lesssim e^{-A |z|^{1/s}}, \quad z \in \rr {2d}, \quad A>0,
\end{equation}
we obtain by means of dominated convergence and \eqref{STFTrecon2}
\begin{align*}
(u_n,\psi) & = (2 \pi)^{-d} \int_{|z| \leqslant n} V_\fy u (z) \ \overline{V_\fy \psi (z)} \, dz \\
& \longrightarrow  (2 \pi)^{-d} \int_{\rr {2d}} V_\fy u(z) \ \overline{V_\fy \psi (z)} \, dz \\
& = (u,\psi), \quad n \rightarrow \infty. 
\end{align*}
This proves the claim that $u_n \rightarrow u$ in $\Sigma_s'(\rr d)$ as $n \rightarrow \infty$. 

We also need the estimate (cf. \cite[Eq.~(11.29)]{Grochenig1})
\begin{equation*}
|V_{\overline{\fy}} u_n (z)| \leqslant (2 \pi)^{-d} |V_\varphi u| * |V_{\overline{\fy}} \fy| (z), \quad z \in \rr {2d}, 
\end{equation*}
which in view of \eqref{STFTupperbound1} and \eqref{STFTupperbound2} with $\psi$ replaced by $\fy$ and conjugation of the window function gives the bound
\begin{equation}\label{STFTupperbound3}
|V_{\overline{\fy}} u_n (z)| \lesssim e^{4M |z|^{1/s}}, \quad z \in \rr {2d}, \quad n \in \no, 
\end{equation}
that holds uniformly over $n \in \no$. 

We are now in a position to assemble the arguments into a proof of formula 
\eqref{TSTFT} for $u \in \Sigma_s'(\rr d)$ and $\psi \in \Sigma_s(\rr d)$. 
Using \eqref{TSTFT} for $u_n$ gives
\begin{equation}\label{STFTequalitylimit}
(\cK u, \psi) 
= \lim_{n \rightarrow \infty} 
(2 \pi)^{-2d} \int_{\rr {4d}} V_\Phi K(x,y,\xi,-\eta) \, \overline{V_\fy \psi (x,\xi)} \, V_{\overline \fy} u_n(y,\eta) \, dx \, dy \, d \xi \, d \eta.
\end{equation}
Since $V_{\overline \fy} u_n(y,\eta) \rightarrow V_{\overline \fy} u(y,\eta)$ as $n \rightarrow \infty$ for all $(y,\eta) \in \rr {2d}$, the formula \eqref{TSTFT} follows from dominated convergence if we can show that the modulus of the integrand in \eqref{STFTequalitylimit} is bounded by an integrable function that does not depend on $n \in \no$. 

For $C > 1$ define the open conic set
\begin{equation*}
\Gamma = \{ (x,y,\xi,\eta) \in T^* \rr {2d}: \ C^{-1} |(x,\xi)| <  |(y,\eta)| < C |(x,\xi)| \} \subseteq T^* \rr {2d} \setminus 0. 
\end{equation*}
If $C$ is chosen properly then we have $WK^s (K) \subseteq \Gamma$ by  Lemma \ref{WFkernelaxes}. 

Let $\ep>0$. 
For the integral in \eqref{STFTequalitylimit} over $\Gamma'$ we may estimate, using the estimate
\begin{equation}\label{kernelSTFT}
|V_\Phi K(x,y,\xi,-\eta)| \lesssim e^{B |(x,y,\xi,\eta)|^{1/s}}, \quad (x,y,\xi,\eta) \in \rr {4d}, 
\end{equation}
for some $B \geqslant 0$ (cf. \eqref{STFTgrowth}), and \eqref{STFTupperbound2}, \eqref{STFTupperbound3}, for any $A>0$
\begin{equation}\label{integralGamma}
\begin{aligned}
& \int_{\Gamma'} |V_\Phi K(x,y,\xi,-\eta)| \,|V_\fy \psi (x,\xi)| \, |V_{\overline \fy} u_n(y,\eta)| \, dx \, dy \, d \xi \, d \eta \\
& \lesssim \int_{\Gamma'} e^{-\ep |(y,\eta)|^{1/s}} e^{(2B-A) |(x,\xi)|^{1/s} + (2B+4M +\ep) |(y,\eta)|^{1/s}}  \, dx \, dy \, d \xi \, d \eta \\
& \leqslant \int_{\Gamma'} e^{-\ep |(y,\eta)|^{1/s}} e^{(2B-A+C^{1/s}(2B+4M +\ep)) |(x,\xi)|^{1/s}} \, dx \, dy \, d \xi \, d \eta 
< \infty, 
\end{aligned}
\end{equation}
in the final inequality assuming that $A>0$ is sufficiently large.

Since $\Gamma \subseteq T^* \rr {2d} \setminus 0$ is open and $WK^s (K) \subseteq \Gamma$  
we have for any $A>0$
\begin{equation*}
|V_\Phi K(x,y,\xi,-\eta)| \lesssim e^{-A |(x,y,\xi,\eta)|^{1/s}}, \quad (x,y,\xi,-\eta) \in \rr {4d} \setminus \Gamma. 
\end{equation*}
This gives for the integral in \eqref{STFTequalitylimit} over $\rr {4d} \setminus \Gamma'$, again using \eqref{STFTupperbound3}, 
\begin{equation}\label{integralGammacomp}
\begin{aligned}
& \int_{\rr {4d} \setminus \Gamma'} |V_\Phi K(x,y,\xi,-\eta)| \,|V_\fy \psi (x,\xi)| \, |V_{\overline \fy} u_n(y,\eta)| \, dx \, dy \, d \xi \, d \eta \\
& \lesssim \int_{\rr {4d} \setminus \Gamma'} e^{-A |(x,y,\xi,\eta)|^{1/s} + 4M |(y,\eta)|^{1/s} } \, dx \, dy \, d \xi \, d \eta \\
& \leqslant \int_{ \rr {4d} } e^{ (4M-A) |(x,y,\xi,\eta)|^{1/s} } \, dx \, dy \, d \xi \, d \eta  < \infty
\end{aligned}
\end{equation}
provided $A>0$ is sufficiently large.

Combined, \eqref{integralGamma} and \eqref{integralGammacomp} show our claim that the modulus of the integrand in \eqref{STFTequalitylimit}
is bounded by an integrable function that does not depend on $n \in \no$. 
\end{proof}

Since
\begin{equation*}
\overline{V_\fy \Pi(t,\theta) \fy (x,\xi)} = e^{i \la x, \xi - \theta \ra} V_\fy \fy (t-x,\theta-\xi), \quad t,x,\theta,\xi \in \rr d, 
\end{equation*}
we obtain from Lemma \ref{STFTkernelformula} with $\psi = \Pi(t,\theta) \fy$ for $(t,\theta) \in \rr {2d}$, 
$u \in \Sigma_s'(\rr d)$, $\fy \in \Sigma_s(\rr d)$ and $\| \fy \|_{L^2} = 1$
\begin{equation}\label{STFTop1}
\begin{aligned}
& V_\fy(\cK u) (t, \theta) 
= (\cK u, \Pi(t,\theta) \fy) \\
& = (2 \pi)^{-2d} \int_{\rr {4d}} e^{i \la x,\xi-\theta \ra} V_\Phi K (x,y,\xi,-\eta) V_\fy \fy (t-x,\theta-\xi) \, V_{\overline \fy} u(y,\eta) \, dx \, dy \, d \xi \, d \eta. 
\end{aligned}
\end{equation}
This formula will be useful in the proof of Theorem \ref{WFphaseincl}. 

The following result concerns propagation of singularities for linear operators and 
is a version of H\"ormander's \cite[Proposition~2.11]{Hormander1} adapted to the $s$-Gelfand--Shilov wave front set. 
We use the relation mapping notation 
\begin{align*}
& WF^s(K)' \circ WF^s (u) \\
& = \{ (x,\xi) \in T^* \rr d: \,  \exists (y,\eta) \in WF^s (u) : \, (x,y,\xi,-\eta) \in WF^s(K) \}. 
\end{align*}

\begin{thm}\label{WFphaseincl}
Let $s>1/2$ and let $\cK$ be the continuous linear operator \eqref{kernelop}
defined by the Schwartz kernel $K \in \Sigma_s'(\rr {2d})$. 
Suppose $\cK: \Sigma_s(\rr d) \rightarrow \Sigma_s(\rr d)$ is continuous and extends uniquely to a continuous linear operator $\cK: \Sigma_s' (\rr d) \rightarrow \Sigma_s' (\rr d)$, 
and suppose 
\begin{equation}\label{WKjempty}
WF_1^s(K) = WF_2^s(K) = \emptyset.
\end{equation}  
Then for $u \in \Sigma_s'(\rr d)$
we have
\begin{equation}\label{WFphaseinclusion}
WF^s (\cK u) \subseteq WF^s(K)' \circ WF^s (u).  
\end{equation}
\end{thm}

\begin{proof}
It follows from \eqref{STFTgrowth} that \eqref{kernelSTFT} is satisfied for some $B \geqslant 0$ if $\Phi \in \Sigma_s(\rr {2d}) \setminus 0$. 

Denote by 
\begin{align*}
p_{1,3}(x,y,\xi,\eta) & = (x,\xi), \\
p_{2,-4}(x,y,\xi,\eta) & = (y,-\eta), \quad x,y,\xi, \eta \in \rr d, 
\end{align*}
the projections $\rr {4d} \rightarrow \rr {2d}$ onto the first and the third $\rr d$ coordinate, 
and onto the second and the fourth $\rr d$ coordinate with a change sign in the latter, respectively. 

By Lemma \ref{WFkernelaxes} there exists $c>1$ such that 
\begin{equation*}
WF^s (K) \subseteq \Gamma_1 = \{ (x,y,\xi,\eta) \in T^* \rr {2d}: \ c^{-1} |(x,\xi)| <  |(y,\eta)| < c |(x,\xi)| \}, 
\end{equation*}
and, defining
\begin{align*}
\Gamma_{1,3} & = \{(x,y,\xi,\eta) \in T^* \rr {2d}: \ c |(x,\xi)| \leqslant |(y,\eta)| \}, \\
\Gamma_{2,4} & = \{(x,y,\xi,\eta) \in T^* \rr {2d}: \ c |(y,\eta)| \leqslant  |(x,\xi)| \}, \\
\end{align*}
we thus have
\begin{equation}\label{inclusionG1}
\Gamma_1 \subseteq \rr {4d} \setminus (\Gamma_{1,3} \cup \Gamma_{2,4} ). 
\end{equation}

We show the inclusion  \eqref{WFphaseinclusion} by showing that 
\begin{equation}\label{assumption1}
0 \neq (t_0,\theta_0) \notin WF^s(K)' \circ WF^s (u) 
\end{equation}
implies $(t_0,\theta_0) \notin WF^s (\cK u)$.
Thus we suppose \eqref{assumption1}. 
By \cite[Lemma~4.2]{Wahlberg1} we may assume that $(t_0,\theta_0) \in \Omega_0$ and $\overline \Omega_0 \cap WF^s(K)' \circ \overline \Omega_2 = \emptyset$
where $\Omega_0, \Omega_2 \subseteq T^* \rr d \setminus 0$ are conic, open and  $WF^s (u) \subseteq \Omega_2$. 
(The assumption of \cite[Lemma~4.2]{Wahlberg1} corresponds to the assumption \eqref{WKjempty}.)
Here we use the notation $\overline \Omega \subseteq T^* \rr d \setminus 0$ for the closure  in the usual topology in $T^* \rr d \setminus 0$ of a conical subset $\Omega \subseteq T^* \rr d \setminus 0$. 

Hence
\begin{equation*}
\overline \Omega_0 \cap p_{1,3} \left( WF^s(K) \cap p_{2,-4}^{-1} \, \overline \Omega_2 \right) = \emptyset, 
\end{equation*}
or, equivalently, 
\begin{equation*}
p_{1,3}^{-1} \, \overline{\Omega}_0 \cap WF^s(K) \cap p_{2,-4}^{-1} \, \overline \Omega_2 = \emptyset.  
\end{equation*}

Due to assumption \eqref{WKjempty} we may strengthen this into 
\begin{equation*}
p_{1,3}^{-1} \, (\overline{\Omega}_0 \cup \{ 0 \} ) \setminus 0 \cap WF^s(K) \cap p_{2,-4}^{-1} \, (\overline \Omega_2 \cup \{ 0 \} ) \setminus 0 = \emptyset.  
\end{equation*}
Since $p_{1,3}^{-1} \, (\overline{\Omega}_0 \cup \{ 0 \} ) \setminus 0$ and $p_{2,-4}^{-1} \, (\overline \Omega_2 \cup \{ 0 \} ) \setminus 0$ are closed conic subsets of $\rr {4d} \setminus 0$, 
decreasing $\Gamma_1 \subseteq \rr {4d} \setminus 0$ if necessary, 
there exist open conic subsets $\Gamma_0, \Gamma_1, \Gamma_2 \subseteq \rr {4d} \setminus 0$ such that 
\begin{equation*}
WF^s(K) \subseteq \Gamma_1, \qquad
p_{1,3}^{-1} \, \overline{\Omega}_0 \subseteq \Gamma_0, \qquad
p_{2,-4}^{-1} \, \overline{\Omega}_2 \subseteq \Gamma_2, 
\end{equation*}
and
\begin{equation}\label{intersection1}
\Gamma_0 \cap \Gamma_1 \cap \Gamma_2 = \emptyset. 
\end{equation}

Let $\Sigma_0 \subseteq T^* \rr d \setminus 0$ be an open conic set such that 
$(t_0,\theta_0) \in \Sigma_0$ and $\overline{\Sigma_0 \cap S_{2d-1}} \subseteq \Omega_0$. 
Let $\fy \in \Sigma_s(\rr d)$, $\| \fy \|_{L^2} = 1$ and $\Phi = \fy \otimes \fy$.
From Lemma \ref{STFTkernelformula} we know that formula \eqref{STFTop1} holds. 
Therefore we have for any $A>0$
\begin{equation}\label{estimand1}
\begin{aligned}
& e^{A |(t,\theta)|^{1/s}} |V_\fy(\cK u) (t, \theta)| \\
& \lesssim
\int_{\rr {4d}} |V_\Phi K(x,y,\xi,-\eta)| \, e^{A |(t,\theta)|^{1/s}} \, | V_\fy \fy (t-x,\theta-\xi)| 
\, |V_{\overline \fy} u(y,\eta)| \, dx \, dy \, d \xi \, d \eta. 
\end{aligned}
\end{equation}
We will show that this integral is bounded when $(t,\theta) \in \Sigma_0$ for any $A>0$ which proves that $(t_0,\theta_0) \notin WF^s (\cK u)$. 

Consider first the right hand side integral over $(x,y,\xi,-\eta) \in \rr {4d} \setminus \Gamma_1$. 
We have for any $b>0$
\begin{equation}\label{WFcompl}
|V_\Phi K(x,y,\xi,-\eta)| \lesssim e^{-b |(x,y,\xi,\eta)|^{1/s}} , \quad (x,y,\xi,-\eta) \in \rr {4d} \setminus \Gamma_1, 
\end{equation}
due to $WF^s(K) \subseteq \Gamma_1$ and $\Gamma_1 \subseteq \rr {4d} \setminus 0$ being open. 
By \eqref{STFTgrowth} we have 
\begin{equation}\label{STFTu}
|V_\varphi u (z)| \lesssim e^{M |z|^{1/s}}, \quad z \in \rr {2d}, 
\end{equation}
for some $M \geqslant 0$, and by Proposition \ref{seminormequivalence} we have 
\begin{equation}\label{STFTphi}
|V_\varphi \varphi (z)| \lesssim e^{-c |z|^{1/s}}, \quad z \in \rr {2d}, 
\end{equation}
for any $c>0$. 
Thus 
\begin{equation}\label{estimateA}
\begin{aligned}
& \int_{\rr {4d} \setminus \Gamma_1'} |V_\Phi K(x,y,\xi,-\eta)| \, e^{A |(t,\theta)|^{1/s}}  \, | V_\fy \fy (t-x,\theta-\xi)| 
\, |V_{\overline \fy} u(y,\eta)| \, dx \, dy \, d \xi \, d \eta \\
& \lesssim \int_{\rr {4d} \setminus \Gamma_1'} e^{-b |(x,y,\xi,\eta)|^{1/s} + 2 A |(x,\xi)|^{1/s} + M |(y,\eta)|^{1/s}} \, e^{2A |(t-x,\theta-\xi)|^{1/s}} \, | V_\fy \fy (t-x,\theta-\xi)| \, dx \, dy \, d \xi \, d \eta \\
& \lesssim \int_{\rr {4d}} e^{(2 A + M - b) |(x,y,\xi,\eta)|^{1/s}} \, dx \, dy \, d \xi \, d \eta
< \infty
\end{aligned}
\end{equation}
if $b > 0$ is chosen sufficiently large. The estimate holds for all $(t,\theta) \in \rr {2d}$.  

It remains to estimate the right hand side integral \eqref{estimand1} over $(x,y,\xi,-\eta) \in \Gamma_1$. 
By \eqref{inclusionG1} and \eqref{intersection1} we have $\Gamma_1 \subseteq G_1 \cup G_2$ where
\begin{equation*}
G_1 = \rr {4d} \setminus (\Gamma_{1,3} \cup \Gamma_{2,4} \cup \Gamma_0), \quad G_2 = \rr {4d} \setminus (\Gamma_{1,3} \cup \Gamma_{2,4} \cup \Gamma_2 ). 
\end{equation*}
When $(x,y,\xi,-\eta) \in \Gamma_1$ we have $|(x,\xi)|\asymp |(y,\eta)|$.  
First we study $(x,y,\xi,-\eta) \in G_1$. Then $(x,y,\xi,-\eta) \notin \Gamma_0$ which implies $(x,\xi) \notin \Omega_0$. 
There exists $\delta>0$ such that
\begin{equation*}
|(x,\xi) - (t,\theta)| \geqslant \delta |(x,\xi)|, \quad (x,\xi) \notin \Omega_0, \quad (t,\theta) \in \Sigma_0. 
\end{equation*}

Let $\ep>0$. 
For $(t,\theta) \in \Sigma_0$ we obtain with the aid of \eqref{kernelSTFT}
using $|(x,\xi)|\asymp |(y,\eta)|$, \eqref{STFTu} and \eqref{STFTphi}, where $B,M \geqslant 0$ are fixed and $c>0$ is arbitrary, with $B_1>0$ a new constant that depends on $B, M, \ep, s$, 
\begin{equation}\label{estimateB}
\begin{aligned}
& \int_{G_1'} |V_\Phi K(x,y,\xi,-\eta)| \, e^{A |(t,\theta)|^{1/s}} \,  | V_\fy \fy (t-x,\theta-\xi)| \, 
|V_{\overline \fy} u(y,\eta)| \, dx \, dy \, d \xi \, d \eta \\
& \lesssim \int_{G_1'} e^{B |(x,y,\xi,\eta)|^{1/s} + 2A |(x,\xi)|^{1/s} + M |(y,\eta)|^{1/s} } \\
& \qquad \qquad \qquad \qquad \qquad  \times  e^{2A |(t-x,\theta-\xi)|^{1/s}}  \,  | V_\fy \fy (t-x,\theta-\xi)| \, dx \, dy \, d \xi \, d \eta \\
& \lesssim \int_{G_1'} e^{-\ep |(x,y,\xi,\eta)|^{1/s} + B_1 |(x,\xi)|^{1/s} - c |(t-x,\theta-\xi)|^{1/s} } \, dx \, dy \, d \xi \, d \eta \\
& \lesssim \int_{\rr {4d}} e^{-\ep |(x,y,\xi,\eta)|^{1/s} + (B_1- c \delta^{1/s}) |(x,\xi)|^{1/s} } \, dx \, dy \, d \xi \, d \eta \\
& < \infty
\end{aligned}
\end{equation}
provided $c \geqslant B_1 \delta^{-1/s}$. 

Finally we study $(x,y,\xi,-\eta) \in G_2$. Then $(x,y,\xi,-\eta) \notin \Gamma_2$ so we have $(y,\eta) \notin \Omega_2$. 
Hence $(y,\eta) \in G$ where $G \subseteq T^* \rr d$ is closed, conic and does not intersect $WF^s(u)$. 
We obtain with the aid of \eqref{kernelSTFT} for any $(t,\theta) \in T^* \rr d$,
using $|(x,\xi)|\asymp |(y,\eta)|$ and \eqref{STFTphi}, 
for $B \geqslant 0$ fixed and $B_1>0$ a new constant that depends on $A, B, \ep,  s$, 
\begin{equation}\label{estimateC}
\begin{aligned}
& \int_{G_2'} |V_\Phi K(x,y,\xi,-\eta)| \, e^{A |(t,\theta)|^{1/s}} \,  | V_\fy \fy (t-x,\theta-\xi)| 
\, |V_{\overline \fy} u(y,\eta)| \, dx \, dy \, d \xi \, d \eta \\
& \lesssim \int_{G_2'} e^{B |(x,y,\xi,\eta)|^{1/s} + 2A |(x,\xi)|^{1/s}} \, |V_{\overline \fy} u(y,\eta)|  \, dx \, dy \, d \xi \, d \eta \\
& \lesssim \int_{G_2'} e^{-\ep |(x,y,\xi,\eta)|^{1/s} + B_1 |(y,\eta)|^{1/s}} \, |V_{\overline \fy} u(y,\eta)|  \, dx \, dy \, d \xi \, d \eta \\
& \lesssim \sup_{w \in G} \,e^{B_1 |w|^{1/s}} \, |V_{\overline \fy} u(w)|
< \infty. 
\end{aligned}
\end{equation}

We can now combine \eqref{estimand1}, \eqref{estimateA}, $\Gamma_1 \subseteq G_1 \cup G_2$, \eqref{estimateB} and \eqref{estimateC} to conclude
\begin{equation*}
\sup_{(t,\theta) \in \Sigma_0} e^{A |(t,\theta)|^{1/s}} \, |V_\fy (\cK u) (t,\theta)| < \infty
\end{equation*}
for any $A>0$.
Thus $(t_0,\theta_0) \notin WF^s (\cK u)$. 
\end{proof}

\section{The $s$-Gelfand--Shilov wave front set of oscillatory integrals}\label{sec:oscint}

\subsection{Oscillatory integrals}\label{secoscint_one}

We need to describe a class of oscillatory integrals with quadratic phase functions introduced by H\"ormander \cite{Hormander2}. 
This is useful due to the fact that the Schwartz kernel of the Schr\"odinger propagator $e^{-t q^w(x,D)}$ is an integral of this form, as we will describe in Section \ref{sec:kernelschrod}. 
Our discussion on oscillatory integrals is brief. For a richer account we refer to \cite{Hormander2,Rodino2}. 

Let $p$ be a complex-valued quadratic form on $\rr {d + N}$,
\begin{equation}\label{pform}
p(x,\theta) = \la (x, \theta), P (x, \theta) \ra, \quad x \in \rr d, \quad \theta \in \rr N, 
\end{equation}
where $P \in \cc {(d+N) \times (d+N)}$ is the symmetric matrix
\begin{equation}\label{Pmatrix}
P=\left(
\begin{array}{ll}
P_{xx} & P_{x \theta} \\
P_{\theta x} & P_{\theta \theta} 
\end{array}
\right)
\end{equation}
where $P_{xx} \in \cc {d \times d}$, $P_{x \theta} \in \cc {d \times N}$ and $P_{\theta \theta} \in \cc {N \times N}$. 

Suppose $P$ satisfies the following two conditions. 

\begin{enumerate}

\item  $\im P \geqslant 0$; 

\item the row vectors of  the submatrix
\begin{equation*}
\left(
\begin{array}{ll}
P_{\theta x} & P_{\theta \theta} 
\end{array}
\right) \in \cc {N \times (d+N)}
\end{equation*}
are linearly independent over $\co$. 

\end{enumerate}

Under these circumstances the oscillatory integral
\begin{equation}\label{oscillint1}
u(x) = \int_{\rr N} e^{i p(x,\theta)} \, d \theta, \quad x \in \rr d, 
\end{equation}
can be given a unique meaning as an element in $\cS'(\rr d)$, by means of a regularization procedure \cite{Hormander2,Rodino2}. 
Due to the embedding $\cS'(\rr d) \subseteq \Sigma_s'(\rr d)$ ($s>1/2$), the oscillatory integral defines a unique element 
$u \in \Sigma_s'(\rr d)$. 

An oscillatory integral of the form \eqref{oscillint1} is, up to multiplication with an element in $\co \setminus 0$, bijectively associated with a Lagrangian subspace $\lambda \subseteq T^* \cc d$, that is positive in the sense of 
\begin{equation*}
i \sigma (\overline X, X ) \geqslant 0, \quad X \in \lambda. 
\end{equation*}

The positive Lagrangian associated with the oscillatory integral \eqref{oscillint1} is
\begin{equation}\label{lagrangian1}
\lambda = \{ (x, p_x'(x,\theta) ) \in T^* \cc d: \ p_\theta'(x,\theta) = 0, \ (x,\theta) \in \cc {d+N} \} \subseteq T^* \cc d. 
\end{equation}
The integer $N$ in the integral \eqref{oscillint1} is not uniquely determined by $u$. 
In fact, it may be possible to decrease $N$ and obtain the same $u$ times a nonzero complex constant. 
This procedure may be iterated until the term that is quadratic in $\theta$ disappears. 
The form $p$ is then (cf. \cite[Propositions 5.6 and 5.7]{Hormander2})
\begin{equation}\label{canonicp}
p(x,\theta) = \rho(x) + \la L \theta,x \ra
\end{equation}
where $\rho$ is a quadratic form, $\im \rho \geqslant 0$ and $L \in \rr {d \times N}$ is an injective matrix. 
The matrix $L$ is uniquely determined by the Lagrangian $\lambda$ modulo invertible right factors, 
and similarly the values of $\rho$ on $\Ker L^t$ are uniquely determined.  

The oscillatory integral is (cf. \cite[Proposition 5.7]{Hormander2})
\begin{equation}\label{oscillint2}
u(x) = (2 \pi)^N \delta_0 (L^t x) e^{i \rho(x)}, \quad x \in \rr d, 
\end{equation}
where $\delta_0 = \delta_0(\rr N)$. 

\subsection{The $s$-Gelfand--Shilov wave front set of an oscillatory integral}

\begin{thm}\label{WFsoscint}
Let $u \in \cS'(\rr d)$ be the oscillatory integral \eqref{oscillint1} with the associated positive Lagrangian $\lambda$ given by \eqref{lagrangian1}. 
Then for $s > 1/2$
\begin{equation}\label{WFincl}
WF^s(u) \subseteq (\lambda \cap T^* \rr d) \setminus 0. 
\end{equation}
\end{thm}

\begin{proof}
First we assume $N \geqslant 1$ in \eqref{oscillint1} and in the end we will take care of the case $N=0$. 

As discussed above we may assume that $p$ is of the form \eqref{canonicp}
where $\rho$ is a quadratic form, $\im \rho \geqslant 0$ and $L \in \rr {d \times N}$ is injective,  
and $u$ is given by \eqref{oscillint2}. 
The positive Lagrangian \eqref{lagrangian1} associated to $u$ is hence
\begin{equation}\label{lagrangian3}
\lambda = \{ (x,\rho'(x) + L \theta): \, (x,\theta) \in \cc {d+N}, \, L^t x = 0 \} \subseteq T^* \cc d. 
\end{equation}

By \cite[Proposition~3.4]{Rodino2} and the uniqueness part of \cite[Proposition 5.7]{Hormander2} we may assume 
\begin{equation*}
\re (\Ran \rho') \perp \Ran L, \quad \im (\Ran \rho') \perp \Ran L. 
\end{equation*}
If  $x \in \rr d$, $\theta \in \cc N$ and $0 = \im (\rho'(x) + L \theta) = \im (\rho')(x) + L \im \theta$, 
we may thus conclude $\im (\rho')(x)=0$ and $L \im \theta=0$, so the injectivity of $L$ forces $\theta \in \rr N$. 
Hence 
\begin{equation}\label{lambdareal}
\lambda \cap T^* \rr d = \{ (x,\re(\rho')(x) + L \theta): \, (x,\theta) \in \rr {d+N}, \, L^t x = 0 \}. 
\end{equation}

According to Proposition \ref{linsurj} and \eqref{WFsdirac}
\begin{align*}
WF^s( \delta_0 (L^t \cdot) ) 
& = \{ (x,L \xi) \in T^* \rr d \setminus 0: \, L^t x=0, \, \xi \in \rr N \setminus 0 \} \cup \Ker L^t \setminus 0 \times \{ 0 \} \\
& = (\Ker L^t \times L \rr N) \setminus 0 \subseteq T^* \rr d \setminus 0. 
\end{align*}
We write $\rho(x) = \rho_r(x) + i \rho_i(x)$ and 
\begin{equation}\label{rhodef}
\rho_r (x) = \la R_r x,x \ra, \quad \rho_i (x) = \la R_i x, x \ra
\end{equation}
with $R_r, R_i \in \rr {d \times d}$ symmetric and $R_i \geqslant 0$. 
The function $e^{ i \rho_r(x)}$, considered as a multiplication operator, is the metaplectic operator 
\begin{equation*}
e^{ i \rho_r(x)} = \mu (\chi)
\end{equation*}
where 
\begin{equation}\label{chidef}
\chi = 
\left(
\begin{array}{ll}
I & 0 \\
2 R_r & I
\end{array}
\right) \in \Sp(d,\ro).
\end{equation}

The function $g(x) = e^{- \rho_i(x) }$ satisfies the estimates \eqref{symbolestimate0} for all $h>0$. 
In fact, since $R_i \geqslant 0$ it suffices to verify the estimates for $g(Ux) = \Pi_{j=1}^n e^{-\mu_j x_j^2}$ where $U \in \rr {d \times d}$ is an orthogonal matrix
and $\mu_j>0$ for $1 \leqslant j \leqslant n \leqslant d$. 
The function $x \rightarrow g(Ux)$ clearly satisfies \eqref{symbolestimate0} for all $h>0$ since it is a tensor product of a Gaussian on $\rr n$, that belongs to $\Sigma_s(\rr n)$, and the function one on $\rr {d-n}$. 

If we consider $e^{- \rho_i(x) }$ as a function of $(x,\xi) \in T^* \rr d$, constant with respect to the $\xi$ variable, then 
the corresponding Weyl pseudodifferential operator is multiplication with $e^{- \rho_i(x) }$. 
Proposition \ref{microlocalWFs} gives
\begin{equation}\label{propagation0}
WF^s(e^{- \rho_i } u) \subseteq WF^s(u), \quad u \in \Sigma_s'(\rr d). 
\end{equation}
Piecing these arguments together, using 
Corollary \ref{symplecticWFs}
and \eqref{lambdareal}, gives
\begin{equation}\label{WFincl0}
\begin{aligned}
WF^s( u ) 
& = WF^s ( e^{- \rho_i } e^{i \rho_r} \delta_0 (L^t \cdot) ) \\
& \subseteq WF^s ( e^{i \rho_r} \delta_0 (L^t \cdot) ) \\
& = \chi WF^s ( \delta_0 (L^t \cdot) ) \\
& = \{ (x, 2 R_r x + L \theta):  \, (x,\theta)  \in \rr {d+N}, \, L^t x = 0 \} \setminus 0 \\
& = \{ (x, \rho_r'(x) + L \theta):  \, (x,\theta)  \in \rr {d+N}, \, L^t x = 0 \} \setminus 0 \\
& = (\lambda \cap T^* \rr d) \setminus 0. 
\end{aligned}
\end{equation}
This ends the proof when $N \geqslant 1$. 

Finally, if $N=0$ then the (degenerate) oscillatory integral \eqref{oscillint1} is
\begin{equation*}
u(x) = e^{ i \rho(x)}
\end{equation*} 
where $\rho = \rho_r + i \rho_i$ is given by \eqref{rhodef}
with $R_r, R_i \in \rr {d \times d}$ symmetric and $R_i \geqslant 0$. 
By \eqref{lagrangian1} the corresponding positive Lagrangian is 
\begin{equation}\label{lambda0}
\lambda = \{ (x, 2 (R_r + i R_i) x): \, x \in \cc d \} \subseteq T^* \cc d
\end{equation}  
which gives 
\begin{equation*}
\lambda \cap T^* \rr d = \{ (x, 2 R_r x): \, x \in \rr d \}. 
\end{equation*}  
Since $WF^s(1)= (\rr d \setminus 0) \times \{0\}$ (cf. \eqref{WFsone}) we obtain, 
again using \eqref{propagation0} and recycling the argument above, 
\begin{equation*}
\begin{aligned}
WF^s( u ) 
& = WF^s ( e^{- \rho_i } e^{i \rho_r} ) \\
& \subseteq WF^s ( e^{i \rho_r}  ) \\
& = \chi WF^s ( 1 ) \\
& = \{ (x, 2 R_r x):  \, x  \in \rr d \setminus 0 \}  \\
& = (\lambda \cap T^* \rr d) \setminus 0. 
\end{aligned}
\end{equation*}
\end{proof} 

\begin{rem}
From $WF(u) \subseteq WF^s(u)$ for $u \in \cS'(\rr d)$ and $s>1/2$
it follows that Theorem \ref{WFsoscint} is a sharpening of \cite[Theorem~3.6]{Rodino2}.
\end{rem}

\section{The Schwartz kernel of the Schr\"odinger propagator}\label{sec:kernelschrod}

Let $q$ be a quadratic form on $T^*\rr d$ defined by a symmetric matrix $Q \in \cc {2 d \times 2 d}$ with Hamilton map $F=\J Q$ and $\re Q \geqslant 0$.
According to \cite[Theorem 5.12]{Hormander2} the Schr\"odinger propagator is 
\begin{equation*}
e^{-t q^w(x,D)} = \cK_{e^{-2 i t F}}
\end{equation*}
where $\cK_{e^{-2 i t F}}: \cS(\rr d) \rightarrow \cS'(\rr d) $ is the linear continuous operator with Schwartz kernel 
\begin{equation}\label{schwartzkernel1}
K_T (x,y) = (2 \pi)^{-(d+N)/2} \sqrt{\det \left( 
\begin{array}{ll}
p_{\theta \theta}''/i & p_{\theta y}'' \\
p_{x \theta}'' & i p_{x y}'' 
\end{array}
\right) } \int_{\rr N} e^{i p(x,y,\theta)} d\theta \in \cS'(\rr {2d})
\end{equation}
with $T = e^{-2 i t F}$.
This kernel is an oscillatory with respect to a quadratic form $p$ on $\rr {2 d + N}$ as discussed in Section \ref{secoscint_one}. 
The positive Lagrangian associated to the Schwartz kernel $K_{e^{-2 i t F}}$ is 
\begin{equation}\label{twistgraph}
\lambda = \{ (x, y, \xi, -\eta) \in T^* \cc {2d}: \, (x,\xi) =  e^{-2 i t F} (y,\eta) \}.
\end{equation}
By \cite[Lemma~4.2]{Rodino2} the Lagrangian $\lambda = \lambda_{e^{-2 i t F}}$ given by \eqref{twistgraph} is a positive twisted graph Lagrangian defined by the matrix $e^{-2 i t F} \in \Sp(d,\co)$. 
When a twisted graph Lagrangian defined by a matrix $T \in \Sp(d,\co)$ is positive, also the matrix is called positive \cite{Hormander2}. 
This means
\begin{equation*}
i \left( \sigma(\overline{TX}, TX) - \sigma(\overline{X},X) \right) \geqslant 0, \quad X \in T^* \cc d. 
\end{equation*}

Since $K_{e^{-2 i t F}} \in \cS'(\rr {2d})$ the propagator is a continuous operator $e^{-t q^w(x,D)}: \cS (\rr d) \rightarrow \cS' (\rr d)$. 
We have in fact continuity $e^{-t q^w(x,D)}: \cS (\rr d) \rightarrow \cS (\rr d)$. 
This follows from \cite[Proposition~5.8 and Theorem~5.12]{Hormander2} which says that 
$\cK_T: \cS(\rr d) \rightarrow \cS(\rr d)$ is continuous for any positive matrix $T \in \Sp(d,\co)$.
The next result shows that $\cK_T: \Sigma_s (\rr d) \rightarrow \Sigma_s (\rr d)$ is continuous and $\cK_T$ extends uniquely to a continuous operator  
$\cK_T: \Sigma_s' (\rr d) \rightarrow \Sigma_s' (\rr d)$. 

\begin{prop}\label{KTcontGF}
Suppose $T \in \Sp (d,\co)$ is positive and let $\cK_T: \cS(\rr d) \rightarrow \cS'(\rr d)$ be the 
continuous linear operator having Schwartz kernel $K_T \in \cS'(\rr {2d})$ defined by \eqref{schwartzkernel1}. 
For $s > 1/2$ the operator $\cK_T$ is continuous on $\Sigma_s(\rr d)$ 
and $\cK_T$ extends uniquely to a continuous operator on $\Sigma_s'(\rr d)$. 
\end{prop}

\begin{proof}
Due to the above mentioned continuity of $\cK_T$ on $\cS(\rr d)$, we have for some integer $L \geqslant 0$
\begin{equation*}
\sup_{x \in \rr d} |\cK_T f(x)|
\lesssim \sum_{|\alpha+\beta| \leqslant L} \sup_{x \in \rr d} |x^\alpha D^\beta f(x)|.
\end{equation*}

Using the seminorms \eqref{gfseminorm} it can be seen readily that the operators $f \rightarrow D^\beta f$ and $f \rightarrow x^\alpha f$ are continuous operators on $\Sigma_s(\rr d)$. 
By Proposition \ref{seminormequivalence} we therefore have for any $A>0$
\begin{equation*}
e^{A |x|^{1/s}} | x^\alpha D^\beta f (x)| 
\leqslant \| x^\alpha D^\beta f \|_A'
\lesssim \| f \|_B' + \| \wh f \|_B', \quad |\alpha+\beta| \leqslant L, \quad x \in \rr d, 
\end{equation*}
for some $B>0$, where we use the seminorm \eqref{seminorm1}.  
This gives 
\begin{equation}\label{firstestimate}
\| \cK_T f \|_A'
= \sup_{x \in \rr d} e^{A |x|^{1/s}} |\cK_T f(x)|
\lesssim \| f \|_B' + \| \wh f \|_B'
\end{equation}
which gives a desired continuity estimate for one of the two families of seminorms
$\{\| f \|_A', \ \| \wh f \|_B', \ A,B>0 \}$. 

To prove that $\cK_T$ is continuous on $\Sigma_s(\rr d)$ it remains to estimate $\| \cF (\cK_T f) \|_A'$ for any $A>0$. 
The Fourier transform has Schwartz kernel $K(x,y)=e^{-i \la x, y \ra}$ which is a degenerate oscillatory integral with $N=0$, cf. \eqref{oscillint1}.
The corresponding Lagrangian is by \eqref{lagrangian1}
\begin{align*}
\lambda & = \{ (x,y, - y, -x) \in T^* \cc {2d}: (x,y) \in T^* \cc d \} \\
& = \{ (x,y, \xi, -\eta) \in T^* \cc {2d}: (x,\xi) = \J(y,\eta) \}.
\end{align*}
Due to the uniqueness, modulo multiplication with a nonzero complex number, of the correspondence between oscillatory integrals and Lagrangians, it follows that $\cF = (2 \pi)^{d/2} \mu(\J) = c \cK_{\J}$ for some $c \in \co \setminus 0$. 

We may hence write for some $c \in \co \setminus 0$, using the semigroup property of $T \rightarrow \cK_T$ modulo sign, when $T \in \Sp(d,\co)$ is positive (cf. \cite[Proposition~5.9]{Hormander2})
\begin{align*}
\cF (\cK_T f) 
& = \cF \cK_T \cF \cF^{-1} f 
= c^2 \cK_\J \cK_T \cK_\J \cF^{-1} f \\
& = \pm c^2 \cK_{\J T \J} \cF^{-1} f. 
\end{align*}
Since $\J T \J \in \Sp(d,\co)$ is positive, and since the estimate \eqref{firstestimate} holds for any positive $T \in \Sp(d,\co)$ for some $B>0$, we obtain for any $A>0$
\begin{equation}\label{secondestimate}
\| \cF(\cK_T f) \|_A'
\lesssim \| f \|_B' + \| \wh f \|_B'
\end{equation}
for some $B>0$.  
Combining \eqref{firstestimate} and \eqref{secondestimate} and referring to Proposition \ref{seminormequivalence}, we have proved that $\cK_T$ is continuous on $\Sigma_s(\rr d)$.

Finally we show that $\cK_T$ extends uniquely to a continuous operator on $\Sigma_s'(\rr d)$. 
The formal adjoint of $\cK_T$ is $\cK_{\overline{T}^{-1}}$ which is indexed by the inverse conjugate matrix $\overline{T}^{-1} \in \Sp(d,\co)$ \cite{Hormander2}. 
The positivity of $\overline{T}^{-1}$ is an immediate consequence of the assumed positivity of $T$. 
Thus $\cK_T$ may be defined on $\Sigma_s' (\rr d)$ by
\begin{equation*}
( \cK_T u, \fy ) = ( u, \cK_{\overline{T}^{-1}} \fy ), \quad u \in \Sigma_s' (\rr d), \quad \fy \in \Sigma_s (\rr d), \quad s>1/2, 
\end{equation*}
which gives a uniquely defined extension of $\cK_T$ as a continuous operator on $\Sigma_s' (\rr d)$. 
\end{proof}
 
\begin{cor}\label{propagatorcontGF}
The Schr\"odinger propagator $e^{-t q^w(x,D)}$ has Schwartz kernel $K_{e^{-2 i t F}} \in \cS'(\rr {2d})$. 
For $t \geqslant 0$ and $s>1/2$ it is a continuous operator on $\Sigma_s (\rr d)$,
and it extends uniquely to a continuous operator on $\Sigma_s' (\rr d)$. 
\end{cor}

\section{Propagation of the $s$-Gelfand--Shilov wave front set for Schr\"odinger equations}\label{sec:propsing}

Since the Schwartz kernel of the Schr\"odinger propagator $e^{-t q^w(x,D)}$ is an oscillatory integral 
corresponding to the positive Lagrangian \eqref{twistgraph}, 
an appeal to Theorem \ref{WFsoscint} gives the following result. 
For $s>1/2$ the $s$-Gelfand--Shilov wave front set of the Schwartz kernel $K_{e^{-2 i t F}}$ of the propagator $e^{-t q^w(x,D)}$ for $t \geqslant 0$ obeys the  inclusion
\begin{equation}\label{WFkernel1}
\begin{aligned}
& WF^s( K_{e^{-2 i t F}} ) \\
& \subseteq \{ (x, y, \xi, -\eta) \in T^* \rr {2d} \setminus 0 : \, (x,\xi) =  e^{-2 i t F} (y,\eta), \, \im e^{-2 i t F} (y,\eta) = 0 \}. 
\end{aligned}
\end{equation}

Combining this with Corollary \ref{propagatorcontGF} and Theorem \ref{WFphaseincl} now gives a result on propagation of singularities. 
The assumptions of the latter theorem are satisfied for $\cK = \cK_{e^{-2 i t F}} = e^{-t q^w(x,D)}$, 
since $WF_1^s ( K_{e^{-2 i t F}} ) = WF_2^s ( K_{e^{-2 i t F}} )= \emptyset$ follows from \eqref{WFkernel1} and the invertibility of $e^{-2 i t F} \in \cc {2d \times 2d}$, cf. \eqref{WFproj}.

\begin{cor}\label{propagationsing1}
Suppose $q$ is a quadratic form on $T^*\rr d$ defined by a symmetric matrix $Q \in \cc {2d \times 2d}$, $\re Q \geqslant 0$, $F=\J Q$, 
and $s > 1/2$.
Then for $u \in \Sigma_s'(\rr d)$
\begin{align*}
WF^s (e^{-t q^w(x,D)}u) 
& \subseteq  e^{-2 i t F} \left( WF^s (u) \cap \Ker (\im e^{-2 i t F} ) \right), \quad t \geqslant 0. 
\end{align*}
\end{cor}

As in the proof of \cite[Theorem 5.2]{Rodino2}, the latter inclusion can be sharpened using the semigroup modulo sign property of the propagator \cite{Hormander2}
\begin{equation*}
e^{-(t_1+t_2)q^w(x,D)}=\pm e^{-t_1q^w(x,D)}e^{-t_2q^w(x,D)}, \quad t_1,t_2 \geqslant 0. 
\end{equation*} 

In fact, using this property and some elementary arguments one obtains the inclusions
\begin{align*}
WF^s (e^{-t q^w(x,D)}u) 
& \subseteq \left( e^{2 t \im F} \left( WF^s (u) \cap S \right) \right) \cap S \\
& \subseteq e^{-2 i t F} \left( WF^s (u) \cap \Ker (\im e^{-2 i t F} ) \right)  
\end{align*}
where the singular space 
\begin{equation*}
S=\Big(\bigcap_{j=0}^{2d-1} \Ker\big[\re F(\im F)^j \big]\Big) \cap T^*\rr d \subseteq T^*\rr d 
\end{equation*}
of the quadratic form $q$ plays a crucial role. 

\begin{cor}\label{propagationsing2}
Suppose $q$ is a quadratic form on $T^*\rr d$ defined by a symmetric matrix $Q \in \cc {2d \times 2d}$, $\re Q \geqslant 0$, $F=\J Q$,  
and $s > 1/2$.
Then for $u \in \Sigma_s'(\rr d)$
\begin{align*}
WF^s (e^{-t q^w(x,D)}u) 
& \subseteq  \left( e^{2 t \im F} \left( WF^s (u) \cap S \right) \right) \cap S, \quad t > 0. 
\end{align*}
\end{cor}


\end{document}